\documentclass[10pt]{amsart}

\usepackage{amssymb}
\usepackage[all]{xy}
\usepackage{mathrsfs}
\usepackage{enumerate}

\usepackage{bbm}

\newcommand{\rar}[1]{\stackrel{#1}{\longrightarrow}}
\newcommand{\lar}[1]{\stackrel{#1}{\longleftarrow}}

\newcommand{\bA}{{\mathbb A}}

\newcommand{\bC}{{\mathbb C}}
\newcommand{\bD}{{\mathbb D}}

\newcommand{\bP}{{\mathbb P}}
\newcommand{\bQ}{{\mathbb Q}}
\newcommand{\bR}{{\mathbb R}}

\newcommand{\bZ}{{\mathbb Z}}

\newcommand{\cA}{{\mathcal A}}
\newcommand{\cB}{{\mathcal B}}
\newcommand{\cC}{{\mathcal C}}
\newcommand{\cD}{{\mathcal D}}
\newcommand{\cE}{{\mathcal E}}
\newcommand{\cF}{{\mathcal F}}
\newcommand{\cG}{{\mathcal G}}
\newcommand{\cH}{{\mathcal H}}

\newcommand{\cM}{{\mathcal M}}
\newcommand{\cN}{{\mathcal N}}

\newcommand{\cP}{{\mathcal P}}

\newcommand{\cS}{{\mathcal S}}
\newcommand{\cT}{{\mathcal T}}
\newcommand{\cU}{{\mathcal U}}
\newcommand{\cV}{{\mathcal V}}

\newcommand{\cX}{{\mathcal X}}

\newcommand{\eff}{\text{eff}}
\newcommand{\et}{\text{et}}
\newcommand{\gm}{\text{gm}}
\newcommand{\Ho}{\text{Ho}}
\newcommand{\colim}{\text{colim}}

\newcommand{\fg}{{\mathfrak g}}

\newcommand{\nc}{\newcommand}

\nc\wh{\widehat}

\nc\on{\operatorname}

\nc\Gr{\on{Gr}}

\nc\Fl{\on{Fl}}

\newcommand{\limto}{{\displaystyle\lim_{\longrightarrow}}}
\newcommand{\rightlim}{\mathop{\limto}}

\newcommand{\leftlim}{\mathop{\displaystyle\lim_{\longleftarrow}}}
\newcommand{\limfromn}{\leftlim\limits_{\raise3pt\hbox{$n$}}}
\newcommand{\limton}{\rightlim\limits_{\raise3pt\hbox{$n$}}}

\newcommand{\rightlimit}[1]{\mathop{\lim\limits_{\longrightarrow}}\limits%
                    _{\raise3pt\hbox{$\scriptstyle #1$}}}

\newcommand{\leftlimit}[1]{\mathop{\lim\limits_{\longleftarrow}}\limits%
                    _{\raise3pt\hbox{$\scriptstyle #1$}}}

\newcommand{\epi}{\twoheadrightarrow}
\newcommand{\iso}{\buildrel{\sim}\over{\longrightarrow}}
\newcommand{\mono}{\hookrightarrow}

\makeatletter

\newcommand{\Rmnum}[1]{\expandafter\@slowromancap\romannumeral #1@}
\makeatother

\newtheorem{Th}{Theorem}
\newtheorem{pr}{Proposition}[section]
\newtheorem{lm}[pr]{Lemma}

\newtheorem{cor}[pr]{Corollary}

\theoremstyle{definition}

\newtheorem{rem}[pr]{Remark}

\numberwithin{equation}{section}

\begin{document}

\title{Hodge realizations of 1-motives and the derived Albanese}

\author{Vadim Vologodsky}

\address{Department of Mathematics, University of Oregon, Eugene, OR, 97403, USA}
\email{vvologod@uoregon.edu}

\keywords{Algebraic cycles, Voevodsky motives, Hodge theory.}

\subjclass[2000]{Primary 14F42, 14C25;  Secondary 14C22, 14F05.}


\begin{abstract}
We prove that the embedding of the derived category of 1-motives up to isogeny 
into the triangulated category of effective Voevodsky motives, as
well as its left adjoint functor $LAlb_\bQ$, commute with the Hodge
realization. This result yields a new proof of the rational form of Deligne's conjecture on 1-motives. 
\end{abstract}

\maketitle

\section{Introduction}
\subsection{Main results.}\label{m.r.i}  Let  $k$ be a field embedded in $\bC$. We denote by $\cM_1(k)$  the category of Deligne's 1-motives over $k$. 
In (\cite{d3}) Deligne constructed a realization functor from $\cM_1(k)$ to the full subcategory $MHS_1^{\bZ}\subset MHS_{\eff}^\bZ $ of the category 
of effective polarizable mixed Hodge structures\footnote{A mixed Hodge structure is called effective if $F^1=0$.}  consisting of torsion free objects of type $\{(0,0),(0,-1),(-1, 0),(-1,-1)\}$. 
We define exact structures on  $\cM_1(k)$,  $MHS_1^{\bZ}$ postulating that a sequence is exact if its image in the abelian category $MHS_{\eff}^\bZ$ is exact.  Then Deligne's realization functor extends to the corresponding derived categories yielding a triangulated functor
$$T^{Hodge}_\bZ: D^b(  \cM_1(k)) \to  D^b( MHS_1^{\bZ}).$$
On the other hand,  in \cite{bk} Barbieri-Viale and Kahn constructed a fully faithful embedding of the category  $D^b(  \cM_1(k))$ into the triangulated  category of \'etale Voevodsky motives
$$Tot_\bZ : D^b(  \cM_1(k))   \to DM^{\eff}_{\gm, \et}(k; \bZ).$$
In \S \ref{r.v.m} we define a Hodge realization functor
$$R^{Hodge}_\bZ:  DM^{\eff}_{\gm, \et}(k; \bZ) \to  D^b(MHS_{\eff}^\bZ ),$$
which is an integral refinement of the functor constructed by Huber (\cite{hu1}).
 The main result of this paper, conjectured by Barbieri-Viale and Kahn (\cite{bk}, \S 0.13), is the commutativity of the following diagram.
    \begin{equation}\label{mainresintro}
\def\normalbaselines{\baselineskip20pt
\lineskip3pt  \lineskiplimit3pt}
\def\mapright#1{\smash{
\mathop{\to}\limits^{#1}}}
\def\mapdown#1{\Big\downarrow\rlap
{$\vcenter{\hbox{$\scriptstyle#1$}}$}}
\begin{matrix}
D^b(  \cM_1(k)) & \rar{Tot_\bZ} &  DM^{\eff}_{\gm, \et}(k; \bZ) \cr
  \mapdown{T^{Hodge}_\bZ} &  &\mapdown{R^{Hodge}_\bZ}   \cr
 D^b( MHS_1^{\bZ})   &  \rar{\overline {Tot}_\bZ} & D^b(MHS_{\eff}^\bZ ).
\end{matrix}
 \end{equation}
 Here $\overline {Tot}_\bZ$ denotes the functor induced by the inclusion $MHS_1^{\bZ}\subset MHS_{\eff}^\bZ $.
 
To state a corollary of (\ref{mainresintro}), recall  from \cite{bk} that the embedding $$Tot_\bQ: D^b(  \cM_1(k)\otimes \bQ)  \mono DM^{\eff}_{\gm, \et}(k; \bQ)$$
admits a left adjoint  functor, $LAlb_\bQ$,
called  the motivic Albanese functor. The functor  $$\overline {Tot}_\bQ: D^b( MHS_1^{\bQ}) \mono  D^b(MHS_{\eff}^\bQ )$$ also admits a left adjoint  (\cite{bk}, \cite{vol2}), denoted by $\overline{LAlb}_\bQ$.   We prove the commutativity of the following diagram
 \begin{equation}\label{mainresintro2}
\def\normalbaselines{\baselineskip20pt
\lineskip3pt  \lineskiplimit3pt}
\def\mapright#1{\smash{
\mathop{\to}\limits^{#1}}}
\def\mapdown#1{\Big\downarrow\rlap
{$\vcenter{\hbox{$\scriptstyle#1$}}$}}
\begin{matrix}
  DM^{\eff}_{\gm, \et}(k; \bQ)& \rar{LAlb_\bQ}  &D^b(  \cM_1(k)\otimes \bQ)  \cr
 \mapdown{R^{Hodge}_\bQ}  && \mapdown{T^{Hodge}_\bQ }  \cr
 D^b(MHS_{\eff}^\bQ )& \rar{\overline{LAlb}_\bQ } &  D^b( MHS_1^{\bQ}).
\end{matrix}
 \end{equation}
In particular, applying  (\ref{mainresintro2}) to the motive $M(X)$ of a scheme $X$ of finite type over $k$ and writing  $H_i(LAlb_\bQ(X))$ for the $i$-th homology of the complex
$LAlb_\bQ(M(X))\in D^b(  \cM_1(k)\otimes \bQ)$,  
we find a canonical isomorphism
$$T^{Hodge}_\bQ(H_i(LAlb_\bQ(X)))\iso \overline{LAlb} (H_i(X\otimes_k \bC)).$$
Here $\overline{LAlb} (H_i(X\otimes_k \bC))$ denotes the maximal quotient of the mixed $\bQ$-Hodge structure $H_i(X\otimes _k \bC)$ which belongs to $ MHS_1^\bQ $. The fact that 
the mixed Hodge structure $\overline{LAlb} (H_i(X\otimes_k \bC))$ is the realization of a 1-motive defined over $k$ was anticipated by Deligne  (\cite{d3}, \S 10.4.1) and has been known since then as  
 ``Deligne's conjecture on 1-motives up to isogeny''. This conjecture was proven using different methods in  \cite{brs} and in \cite{r2}.
We refer the reader to Part 4 of  \cite{bk} for a detailed discussion of Deligne's conjecture.

In a subsequent paper (\cite{vol3}) we construct a version of (\ref{mainresintro2}) for motives with {\it integral} coefficients. This requires new ideas because  none  of the functors  $\overline{Tot}_\bZ$, $Tot_\bZ$ admits a left adjoint. 
\subsection{Proofs.}\label{rop} 
 The proof of (\ref{mainresintro}) is quite involved. 
In order to compute Huber's realization of 1-motives one has to find a reasonably small functorial resolution for these motives by motives of smooth algebraic  varieties.  This is not trivial already for motives of the form $\underline A$,  where $A$ is an abelian variety.
With rational coefficients a suitable resolution of 1-motives can be produced using the Eilenberg-MacLane cube construction
 (see \S  \ref{e.m.}).\footnote{The idea to use the Eilenberg-MacLane cube construction to construct a nice projective resolution of the presheaf   $\underline A \otimes \bQ$ goes back to the work
of Breen \cite{br}.} Though with integral coefficients  the Eilenberg-MacLane complex has higher cohomology,   in \S \ref{betticompute}-\ref{c.h.r.}  we use its truncated version to define an isomorphism 
\begin{equation}\label{eq104}
(\overline {Tot}_\bZ \circ T^{Hodge}_\bZ)_{\vert \cM_1(k)} \simeq       (R^{Hodge}_\bZ \circ Tot_\bZ) _{\vert \cM_1(k)}.
\end{equation}
The construction of (\ref{eq104}) is the key part of the proof.  As we explain in Remark \ref{wrl}, the commutativity of diagram (\ref{mainresintro}) with {\it rational} coefficients follows trivially from (\ref{eq104}): 
 since $\cM_1(k)\otimes \bQ$ is abelian and has homological dimension $1$,  isomorphism (\ref{eq104}) extends uniquely to an isomorphism of triangulated functors
\begin{equation}\label{eq105}
\overline {Tot}_\bQ \circ T^{Hodge}_\bQ \simeq       R^{Hodge}_\bQ \circ Tot_\bQ.
\end{equation}
In fact,  a  general categorical result  from \cite{vol1} (that we recall in Theorem \ref{bounded}) implies that (\ref{eq104}) extends canonically to $D^b(\cM_1(k))$ provided that
 DG liftings of triangulated categories and functors appearing in diagram (\ref{mainresintro}) are given.
The required work  is done in the lengthy \S \ref{r.v.m}, where we explain
how the Hodge realization functor  $R^{Hodge}_\bZ$ can be lifted to a DG quasi-functor from the DG category of Voevodsky motives (constructed in \cite{bv}) to the derived DG category of Hodge structures.  Although our construction of the DG lifting dwells on known ideas  
from (\cite{bei}, \cite{v2}, \cite{bv})  we believe that it is important on its own ground and has some further applications (see \cite{vol3}). 

 Once the commutative diagram (\ref{mainresintro}) is constructed to prove (\ref{mainresintro2}) it suffices to show that the morphism of functors 
 \begin{equation}\label{eqsecondresin}
  \overline {Tot_\bQ}  \circ \overline{LAlb_\bQ} \circ R^{Hodge}_\bQ
\to R^{Hodge}_\bQ \circ Tot_\bQ  \circ LAlb_\bQ   
\end{equation}
  induced by the canonical morphism  $Id\to Tot_\bQ \circ Alb_\bQ$ is an isomorphism, which is already proven in (\cite{bk}) using the Lefschetz $(1,1)$ Theorem.
   \subsection{A small reading guide.}\label{s.r.g.} Section \ref{r.v.m} contains a construction of DG liftings of the Hodge  and \'etale realization functors as well as definitions and results on 1-motives, which we assume known in the rest of the paper.
  We advise the reader go directly to section \ref{proofs}, which contains the proofs of the results stated above and use \S  \ref{r.v.m}  as a reference. A list of notation can be found at the end of the paper. We would like to warn the reader that symbols 
  ${Tot}_\bZ$,  $LAlb_\bZ$, $R^{Hodge}_\bZ$, {\it etc.},
  are used in the main body of the paper to denote DG liftings (constructed in \S \ref{r.v.m}) of the triangulated functors discussed above.

\subsection{Acknowledgments.}\label{a.} 
 The author is grateful to  Chuck Weibel and the referee(s) for their generous help in turning a raw draft into an article.
 This research was partially supported by NSF grants  DMS-0701106, DMS-0901707.

\section{Realizations of Voevodsky motives.}\label{r.v.m}
In this section we review a DG structure on the category of  \'etale Voevodsky motives and then lift the  Hodge and  \'etale realization
functors  (\cite{hu1}, \cite{hu2}) to DG quasi-functors. Our basic tool is Theorem \ref{abrealization} asserting that a DG realization $D\cM^{\eff}_{\et} (k; A) \to \cC$ into a cocomplete compactly generated DG category $\cC$ is just an ``ordinary'' functor from the category of smooth connected schemes to $\cC$ which is $\bA^1$-homotopy invariant and satisfies the descent property for Voevodsky's $h$-topology. 
    

\subsection{DG categories.}\label{dgcat} The aim of this subsection is to fix notation and to recall some terminology related to the theory of DG categories. We refer the reader to   (\cite{boka}, \cite{dri}, \cite{t}, \cite{k}, \cite{bv} \S1)
for a systematic discussion of the subject.
Fix some universes $\cU\in \cV$.  
$\cU$-small DG categories over commutative ring $A$ can be organized into a 2-category $DGcat_\cU$:  for DG categories $\cC_1$, $\cC_2$, $Mor_{DGcat_\cU}(\cC_1, \cC_2)$ is the category $\cT(\cC_1, \cC_2)$  of DG quasi-functors  (\cite{dri}, \S 16 or \cite{t}, \S 6\footnote{To\"en's notation for $\cT(\cC_1, \cC_2)$ is $\bR\underline{Hom}(\cC_1, \cC_2)$.}). 
Given $\cF \in \cT(\cC_1, \cC_2)$ we will write  $\Ho(\cF): \Ho(\cC_1) \to \Ho(\cC_2)$  for the corresponding functor between the {\it homotopy categories}. 
 Informally, the category of DG quasi-functors is obtained from the category of DG functors by inverting homotopy equivalences: a DG functor $\cF:  \cC_1 \to \cC_2$ is called a homotopy equivalence if  $\Ho(\cF)$ is an equivalence and, for every $X,Y\in \cC_1$, the morphism 
 $Hom_{\cC_1}(X,Y)\to  Hom_{\cC_2}(\cF(X),\cF(Y))$ is a quasi-isomorphism.  If $\cF$ is a homotopy equivalence and $\cC_3$ is another DG category then the composition with $\cF$ induces an equivalence of categories:
 $$\cT(\cC_2, \cC_3)\iso \cT(\cC_1, \cC_3).$$

 A DG $\cU$-category $\cC$ is called {\it $\cU$-cocomplete} if its homotopy category is cocomplete {\it i.e.}, closed under $\cU$-small direct sums.
 For a class $S$ of objects of a triangulated category $D$,  the {\it triangulated subcategory $<S>$ strongly generated by $S$} is the smallest  strictly full triangulated subcategory of $D$ that
contains $S$.  If  $D$ is $\cU$-cocomplete the {\it triangulated subcategory $\langle S \rangle $ generated by $S$}
is the smallest  strictly full triangulated subcategory of $D$ that
contains $S$ and is closed under arbitrary $\cU$-small direct sums.  We will write 
  $D^{perf}\subset D$ for the full subcategory of $\cU$-compact, a.k.a. perfect,  objects of $D$. For a  cocomplete DG category $\cC$  we will denote by
 $\cC^{perf}\subset \cC$ the full DG subcategory that consists of objects compact in $\Ho(\cC)$. A pretriangulated cocomplete  DG $\cU$-category $\cC$ is said to be {\it compactly generated} if  $\Ho(\cC)$ is generated by a $\cU$-small
 set of perfect objects.

 The DG  $\cU$-ind-completion of $\cC \in DGcat_\cU$ will be denoted by $\underrightarrow{\cC}$.  This is a pretriangulated cocomplete $\cV$-small DG $\cU$-category with $\cC \subset \underrightarrow{\cC}^{perf}$
 such that the triangulated category $\Ho(\underrightarrow{\cC})$ is generated by $\Ho(\cC)$ 
 satisfying the following universal property:  
for every pretriangulated cocomplete compactly generated DG $\cU$-category $\cC' \in DGcat_\cV$, one has  
\begin{equation}\label{dgindcompl}
\cT^c(\underrightarrow{\cC}, \cC') \iso \cT(\cC, \cC'),
\end{equation}
where $\cT^c(\underrightarrow{\cC}, \cC')$ is the full subcategory of $\cT(\underrightarrow{\cC}, \cC')$ whose objects are quasi-functors $\cF$ such that $\Ho(\cF)$ commutes with arbitrary $\cU$-small direct sums (\cite{t}, Theorem 7.2\footnote{To\"en's notation for $\underrightarrow{\cC}$ is $\hat{\cC}$.} and \cite{bv}, Proposition 1.7).
The triangulated category
 $\Ho(\underrightarrow{\cC})$ is the derived category of $\cU$-small right DG modules over $\cC$ and will be denoted by  $\bD(\cC)$.

 A {\it DG structure} on a triangulated category $D$ is a pretriangulated DG category $\cC$ together with a triangulated equivalence between $D$ and  $\Ho(\cC)$.  A DG structure on $D$ induces a DG structure on any strictly full triangulated subcategory
$I \subset D$: take for $\cC_I$ the strictly full DG subcategory of $\cC$ whose objects belong to the essential image of the functor $I \mono D \iso \Ho(\cC)$. 
 In particular, for a class $S\subset Ob(D)$ we refer to $\cC_{<S>}\subset \cC$ (resp. $\cC_{\langle S \rangle}\subset \cC$)  as the DG subcategory {\it strongly generated (resp. generated) by $S$}.   

For a DG category $\cC\in DGcat_\cU$ its pretriangulated completion $\cC^{pretr}\in DGcat_\cU$ is a full subcategory of  $\underrightarrow{\cC}$ strongly  generated by  $\Ho(\cC) \subset \Ho(\underrightarrow{\cC})$. The category $\cC^{pretr}$ has the following universal property: for
every pretriangulated DG category $\cC' \in DGcat_\cU$ we have 
\begin{equation}\label{pretr}
\cT(\cC^{pretr}, \cC') \iso \cT(\cC, \cC').
\end{equation}
 
For a full subcategory $\cB \subset \cC\in DGcat_\cU$ the DG quotient $\cC/\cB$ is a $\cU$-small DG category equipped with a DG quasi-functor $\cC \to  \cC/\cB$ satisfying the following universal property: for every DG category $\cC'\in DGcat_\cU$ the functor
\begin{equation}\label{dgquotient}
\cT(\cC/\cB , \cC') \to \cT(\cC, \cC')
\end{equation}
 is fully faithful embedding whose essential image consists of quasi-functors $\cF \in   \cT(\cC, \cC')$ such that $\Ho(\cF)(\Ho(\cB))=0$.\footnote{In particular, if $\cB$ is nonempty and $\Ho(\cC')$ does not have a zero object   $\cT(\cC/\cB , \cC')$ is empty.}
 A DG quotient  $\cC/\cB$ always exists (an unique up to a unique isomorphism in $DGcat_\cU$).  
 In particular, given an $A$-linear exact category  $\cE$  we can define its  bounded derived DG category  $D^{b}_{dg}({\cE})$ to be the DG quotient  of the DG category $C^b_{dg}({\cE})$ of bounded complexes by the subcategory of acyclic ones (\cite{n}, \S 1).

 We will use the following result (\cite{vol1}, Theorem 1).
\begin{Th}\label{bounded}
  Let $\cE$ be a $\cU$-small $A$-linear exact category and $\cE'$ a  $\cU$-small abelian $A$-linear category.  Assume that for every two objects $X,Y\in \cE$,  the $A$-module $Hom(X,Y)$ is flat. 
  \begin{enumerate}
\item{ Let  $$\cF:  D^{b}_{dg}({\cE}) \to D^{b}_{dg}({\cE}') $$ be a DG quasi-functor
 satisfying the following property:

(P) The functors $$H^i \cF: \cE \to D^{b}_{dg}({\cE}) \stackrel{\cF}{\longrightarrow} D^{b}_{dg}({\cE'}) \stackrel{H^i}{\longrightarrow} \cE'$$ are $0$ for every $i<0$ and
  effaceable  ( {\it i.e.},  for every object $X\in \cE$, there is an admissible monomorphism  $X\hookrightarrow Y$ such that 
the induced morphism   $H^i \cF(X) \to  H^i \cF(Y)$ is $0$) for every $i>0$.

Then the functor $F:= H^0 \cF : \cE \to \cE'$ is left exact, has a right derived DG quasi-functor (\cite{dri} \S 5)  
         $$RF:  D^{b}_{dg}({\cE}) \to D^{b}_{dg}({\cE}'), $$ 
         and  there is a unique isomorphism $\cF \simeq RF$ such that the induced automorphism  $F= H^0(\cF) \simeq H^0(RF)= F$ equals $Id$. Conversely,  the right derived DG quasi-functor of any  left exact functor $F:  \cE \to \cE'$ satisfies property (P). } 
         
\item{For every two DG quasi-functors $\cF,\cG  \in \cT(D^{b}_{dg}({\cE}), D^{b}_{dg}({\cE}') ) $ satisfying property (P)
and every $i<0$
$$Hom_{\cT(D^{b}_{dg}({\cE}), D^{b}_{dg}({\cE}') )}(\cF, \cG[i])=0,$$
$$Hom _{\cT(D^{b}_{dg}({\cE}), D^{b}_{dg}({\cE}') )}(\cF, \cG) = Hom_{Fct(\cE, \cE')}(H^0\cF, H^0\cG).$$
Here $Fct(\cE, \cE')$ denotes the category of all $A$-linear functors $\cE \to \cE'$.  }

\end{enumerate}

 \end{Th}

 Having the 2-category $DGcat$ we define the notion of adjoint DG quasi-functors:  given $\cF \in \cT(\cC_1, \cC_2)$  a right  {\it adjunction datum} $(\cG, \nu, \mu)$ consists of 
 a quasi-functor $\cG\in \cT(\cC_2, \cC_1)$ together with  morphisms
 $\nu: Id \to \cG \circ \cF,$ $\mu: \cF \circ \cG \simeq Id $ such that 
 the compositions
$$ \cF \rar{\cF(\nu)}
 \cF\circ \cG \circ \cF
 \rar{\mu(\cF)}     \cF$$
$$ \cG \rar{\nu (\cG)}  
 \cG \circ \cF \circ \cG
 \rar{\cG(\mu) } \cG $$
are identity morphisms.
It is a general property of 2-categories that the adjunction datum if it exists is unique up to a unique isomorphism (\cite{ben}). If this is the case we call $\cG$ the right adjoint quasi-functor. Similarly, one defines the notion of left adjoint quasi-functor.
\begin{lm}\label{dgadj} Let $\cC_1, \cC_2$ be pretriangulated DG categories and let $\cF  \in \cT(\cC_1, \cC_2)$ be a DG quasi-functor. Assume that $\Ho(\cF)$ admits a right (resp. left) adjoint functor. Then
$\cF$ admits a right (resp. left) adjoint quasi-functor.
\end{lm} 
 \begin{proof} For every $\cF\in \cT(\cC_1, \cC_2)$ the induced functor $\cF^*: \underrightarrow{\cC_1} \to \underrightarrow{\cC_2}$ admits a right adjoint quasi-functor
 $\cF_* \in \cT(\underrightarrow{\cC_2}, \underrightarrow{\cC_1})$ (\cite{dri}, \S 14).  If $\Ho(\cG)$ is a right adjoint to $\Ho(\cF)$, for every $X\in \underrightarrow{\cC_2}$, $\Ho(\cF_*)(X)\simeq  \Ho(\cG)(X)$ in $\Ho(\underrightarrow{\cC_1})$. Thus   $ \Ho(\cF_*)$ takes 
 $\Ho(\cC_2)$ into the essential image of $ \Ho(\cC_1)\mono \Ho(\underrightarrow{\cC_1})$. Left adjoint functors are treated similarly.
 \end{proof} 
 
 {\bf Conventions:}  All DG-categories we consider in this paper are of the following 2 types: ($i$)  $\cU$-small DG-categories (small DG categories for short), ($ii$) $\cU$-cocomplete compactly generated $\cV$-small DG $\cU$-categories (cocomplete compactly generated DG categories for short). 
 We will write $DGcat$ for $DGcat_\cU$.  All  fields, rings, modules that we will consider assumed to be $\cU$-small.


\subsection{DG category of \'Etale Voevodsky motives.}\label{dg.mot}
  Let  $A$  a commutative ring and let $k$ be a field of finite \'etale homological dimension with
respect to $A$, {\it i.e.} there exists an integer $N$ such that,
for every discrete $A[Gal(\overline k/k)]$-module $M$ and every
integer $m>N$,
\begin{equation}\label{finiteness}
 H^m(Gal(\overline k/k), M)= 0 .
 \end{equation}

Let $Sm_k$ be the category of smooth varieties over $k$. We will write $A_{tr}[Sm]= A_{tr}[Sm_k]$ for the category whose objects are smooth varieties and morphisms are finite correspondences with coefficients in $A$  (\cite{bv}, \S 2.1.2). 
This an $A$-linear additive category and as such it can be 
viewed as a DG category over $A$.  Let  $PSh_{tr}=PSh_{tr}^A(Sm_k)$  be 
 the (abelian) category  of presheaves of $A$-modules with transfers on
$Sm_k$ {\it i.e.}, the category of contravariant $A$-linear functors $A_{tr}[Sm]\to Mod(A)$,  and let  $D(PSh_{tr})$ be its derived category.  We endow $D(PSh_{tr})$ with a DG structure
$$\Ho(\underrightarrow{A_{tr}[Sm]}) \simeq \bD(A_{tr}[Sm])\simeq D(PSh_{tr}).$$
\begin{rem} 
The Yoneda embedding defines  a homotopy equivalence  $\underrightarrow{A_{tr}[Sm]}\iso D_{dg}(PSh_{tr})$ (\cite{bv}, \S 1.6.4, 1.8).
\end{rem}
 If $X\in Sm_k$  we write $A_{tr}[X]$ for the presheaf represented by $X$.
Consider the following types of complexes in  $PSh_{tr}$: 

 $ (\Delta)$   The 2-term complexes   $A_{tr}[X\times \bA^1] \to A_{tr}[X]$,   where $X\in Sm_k$ and the differential comes from the projection  $X\times \bA^1 \to X$. 

 (Et) The complexes of the form $ A_{tr}[U_{\cdot}] \to A_{tr}[X],$ where $U_{\cdot} \to X$, $U_i\in Sm_k$, is a {\it hypercovering} for the \'etale
topology (\cite{sga4}, \Rmnum{5}  \S 7.3) and  $A_{tr}[U_{\cdot}]$ is the
corresponding normalized complex.  

Let  $I^{ \Delta}_{tr}$, $I^{\et}_{tr}$, $I^{et, \Delta}_{tr}$ be the subcategories of  $\underrightarrow{A_{tr}[Sm]}$ generated by the objects of the corresponding types. We define the {\it  DG category $ D\cM^{\eff}_{\et} (k; A)$ of effective \'etale Voevodsky motives}   to be the DG quotient of 
$\underrightarrow{A_{tr}[Sm]}$ by $I^{et, \Delta}_{tr} $.   The triangulated category $\Ho( D\cM^{\eff}_{\et} (k; A))\simeq D(PSh_{tr})/\Ho(I^{et, \Delta}_{tr} )$ is denoted by $DM^{\eff}_{\et} (k; A)$. By abuse of notation we shall
write $A_{tr}[X]$ for the object of the quotient category $ DM^{\eff}_{\et} (k; A)$
corresponding to the presheaf $A_{tr}[X]$  ({\it i.e.} the motive of
 $X$   ).  The {\it triangulated category $ DM^{\eff}_{\gm, \et}(k; A)$ of effective geometric \'etale  motives} is the smallest idempotent complete strictly full triangulated
subcategory of $ DM^{\eff}_{\et}(k; A)$ that contains all objects of
the form $A_{tr}[X]$, where $X\in Sm_k$.  Being a full subcategory of category equipped with a DG structure $ DM^{\eff}_{\gm, \et}(k; A)$ inherits a DG structure that we denote by $ D\cM^{\eff}_{\gm, \et}(k; A)$.  

The tensor structure on the category $A_{tr}[Sm]$ given by cartesian product of schemes induces a homotopy tensor structure on $ D\cM^{\eff}_{\gm, \et}(k; A)$ , $ D\cM^{\eff}_{\et}(k; A)$ (\cite{bv}, \S 1.9, 2.3.) In particular,   $ DM^{\eff}_{\gm, \et}(k; A)$, $ DM^{\eff}_{\et}(k; A)$ are tensor categories.

\begin{lm}\label{perfectmotives}
$ D\cM^{\eff}_{\et}(k; A)^{perf}=  D\cM^{\eff}_{\gm, \et}(k; A).$
\end{lm}
\begin{proof}  The \'etale coverings define a Grothendieck topology on $A_{tr}[Sm]$   ([BV], \S 4.3).    Let  $Sh^{\et}_{  tr} \subset PSh^{\et}_{  tr}$ be the category of sheaves with transfers.  
    According to ([BV], \S 1.11; 4.3) the sheafification functor $ PSh^{\et}_{  tr} \to  Sh^{\et}_{  tr}$ induces a homotopy equivalence  
         $$D_{dg}(PSh_{tr})/I^{\et}_{tr} \iso D_{dg}(Sh^{\et}_{  tr}).$$
         Hence,
          \begin{equation}\label{eq111}
          D\cM^{\eff}_{\et} (k; A) \iso D_{dg}(Sh^{\et}_{  tr})/I^{\Delta}_{tr}.
          \end{equation}
   Let us check that, for every $X\in Sm_k$, the sheaf  $A_{tr}[X]$ is
 a compact object of $ D(Sh^{\et}_{tr})$ {\it i.e.}, for every $F_i\in   D(Sh^{\et}_{tr})$, $i\in I$, and every integer $m$, the canonical morphism 
 $$  \bigoplus_I Hom_{D(Sh^{\et}_{tr})}(A_{tr}[X], F_i [m])\simeq   \bigoplus R^m\Gamma (X, F_i) \to R^m\Gamma (X, \bigoplus F_i)$$
 is an isomorphism. If each $F_i$ is a single sheaf the claim follows from (\cite{sga4}, \Rmnum{7}, 3.3 ) (and holds for {\it any } field $k$).
 On the other hand, under  (\ref{finiteness}), for every sheaf $F$ of $A$-modules  and every integer $l>N+2dim\, X +1= :N'$,
$H_{\et}^l(X, F)= 0 $ (\cite{sga4}, \Rmnum{10}, 4.3 ) . Thus, replacing each $F_i$ by its canonical
truncation we may assume that $F_i$ is supported in homological
degrees between $m-N' $ and $m$. Our claim follows by d\'evissage.

In particular, the subcategory $\Ho(I^{\Delta}_{tr})\subset  D(Sh^{\et}_{  tr})$ is generated by compact objects. Hence, by  (\cite{bv}, \S 1.4.2, Proposition (ii))  the sheaves $A_{tr}[X]$ remain compact in the quotient category $D(Sh^{\et}_{  tr})/\Ho(I^{\Delta}_{tr})$. Part (i) of the same proposition completes the proof.
    \end{proof}
       \begin{cor}\label{compgen}
 The category  $DM^{\eff}_{\et} (k;  A) $  is compactly generated.
\end{cor}
Indeed, the category  $DM^{\eff}_{\et} (k;  A) $ is generated by
motives of smooth varieties which are compact by the lemma.
\begin{cor}\label{rational} If $(D\cM^{\eff}_{\gm, \et}(k; \bZ) \otimes \bQ)^{\kappa}$ denotes the homotopy idempotent completion (\cite{bv}, \S 1.6.2) of $D\cM^{\eff}_{\gm, \et}(k; \bZ) \otimes \bQ $, the canonical morphism 
$$(D\cM^{\eff}_{\gm, \et}(k; \bZ) \otimes \bQ)^{\kappa} \to D\cM^{\eff}_{\gm, \et}(k; \bQ)$$
is a homotopy equivalence.
\end{cor}
\begin{proof}
It suffices to check that for every $X,Y\in Sm$ the map 
$$Hom_{DM^{\eff}_{\et} (k; \bZ)} ( \bZ_{tr}[X],  \bZ_{tr}[Y]) \otimes \bQ \to   Hom_{DM^{\eff}_{\et}(k; \bZ)} ( \bZ_{tr}[X],  \bQ_{tr}[Y])$$
$$ \iso  Hom_{DM^{\eff}_{ et}(k; \bQ)} ( \bQ_{tr}[X],  \bQ_{tr}[Y])$$
is an isomorphism. Write $\bQ_{tr}[Y]$ as the direct limit of the diagram $\bZ_{tr}[Y]\rar{2}\bZ_{tr}[Y]\rar{3}\bZ_{tr}[Y] \rar{4}\cdots$. Since $\bZ_{tr}[X]$ is perfect,  $Hom_{DM^{\eff}_{\et}(k; \bZ)} ( \bZ_{tr}[X],  \bQ_{tr}[Y])$
is the direct limit of $$Hom_{DM^{\eff}_{\et}(k; \bZ)} ( \bZ_{tr}[X],  \bZ_{tr}[Y])\rar{2} Hom_{DM^{\eff}_{\et}(k; \bZ)} ( \bZ_{tr}[X],  \bZ_{tr}[Y])\rar{3} \cdots .$$
  \end{proof}
  \begin{rem} Voevodsky's category  $DM^{\eff}_{-, et} (k;  A) $ is the image of $D^-(Sh^{\et}_{  tr})$  in  $DM^{\eff}_{\et} (k;  A) $.
\end{rem}
Replacing in the definition of 
$ D\cM^{\eff}_{\et}(k; A)$ \'etale hypercoverings by Nisnevich ones we obtain {\it the DG category $ D\cM^{\eff}(k; A)$ of effective Voevodsky motives} (\cite{bv}). The DG quasi-functor
  $$ D\cM^{\eff}(k; A)\to  D\cM^{\eff}_{\et}(k; A)$$ is a homotopy  equivalence for every $A\supset \bQ$. This is derived from the equivalence $Sh_{tr}^{et \, A} (Sm_k) \iso  Sh_{tr}^{Nis \, A} (Sm_k)$ (\cite{v1}, \S 3.3).   
\subsection{Base change.}\label{bcg}
Let $k\subset k'$ be a field extension. The functor
\begin{equation}\label{basechange}
A_{tr}[Sm_k] \to A_{tr}[Sm_{k'}]
\end{equation}
that takes a smooth $k$-scheme $X$ to $k'$-scheme $X\times _{spec\, k} spec \, k'$ induces DG quasi-functors
\begin{equation}\label{pb}
f^*:  D\cM^{\eff}(k; A)\to D\cM^{\eff}(k'; A), \quad
f^*:  D\cM^{\eff}_{\et} (k;  A) \to D\cM^{\eff}_{\et} (k';  A). 
\end{equation}
We want to describe explicitly the effect  $f^*$ on a complex of presheaves with transfers.   
Consider the left Kan extension of (\ref{basechange})
$$f^{-1}: PSh_{tr}(Sm_k) \to  PSh_{tr}(Sm_{k'}).$$
If  $E\in PSh_{tr}(Sm_k)$ is a presheaf  on the category $A_{tr}[Sm_k]$,
$$ f^{-1} E(X) = \underset {(Y, g\in cor_k(X, Y) )} {\colim}  F(Y),$$
where $X$ is a smooth scheme over $k'$ and the colimit is taken over the category ${\cC}$ all  pairs  $(Y, g\in cor_k(X, Y)= cor_{k'}(X, Y\times _{spec\, k} spec \, k' ))$ with $Y\in A_{tr}[Sm_{k'}]$; $Hom_{{\cC}}((Y, g), (Y', g'))= \{ h \in cor_k(Y, Y'): h\circ g = g' \}$.  One can easily check that the category ${\cC}$  is cofiltrant (\cite{ks}, \S 3.1). 
As a consequence, the functor $f^{-1}$ is exact.
\begin{rem} The subcategory  ${\cC}' \subset {\cC}$ of pairs $(Y, g)$ where $Y \in Sm_k$, $g$ is a morphism $X\to Y$ of $k$-schemes,  and
$Hom_{{\cC}'}((Y, g), (Y', g'))= \{ h \in Mor_k(Y, Y'): h\circ g = g' \}$
 is co-cofinal in  ${\cC}$ ([KS], \S 2.5).  In particular,
 $$ f^{-1} E(X)= \underset {g: X \to Y } {\colim}  F(Y). $$
\end{rem}
As an exact functor $ f^{-1}$ extends in the obvious way to a DG quasi-functor between the DG derived categories of presheaves and this extension fits into the following commutative diagram
    $$
\def\normalbaselines{\baselineskip20pt
\lineskip3pt  \lineskiplimit3pt}
\def\mapright#1{\smash{
\mathop{\to}\limits^{#1}}}
\def\mapdown#1{\Big\downarrow\rlap
{$\vcenter{\hbox{$\scriptstyle#1$}}$}}
\begin{matrix}
D\cM^{\eff}_{\circ} (k;  A) &\lar{}&  \underrightarrow{A_{tr}[Sm_k]}  &    \iso &  D_{dg}(PSh_{tr}(Sm_{k})) \cr
  \mapdown{ f^*}&& \mapdown{ } & &\mapdown{f^{-1}} \cr
D\cM^{\eff}_{\circ} (k';  A) &\lar{}& \underrightarrow{A_{tr}[Sm_{k'}]}  &    \iso &  D_{dg}(PSh_{tr}(Sm_{k'})),
\end{matrix}
$$
where $\circ = et$ or $Nis$, the left vertical arrow is induced by (\ref{basechange}), and the horizontal homotopy equivalences are induced by the Yoneda embeddings.  

 \subsection{h-topology.}\label{h.t}  We shall recall below  a different description of the category $D\cM^{\eff}_{\et} (k; A)$
that does not involve finite correspondences.
 This description will be used  in our construction of the realization functors. In the rest of this section  $char\, k =0$. 

  Let $A[Sm_k]=A[Sm] $ be the category, whose objects are smooth varieties over $k$  and whose morphisms are defined by the formula
       $$Hom_{ A[Sm]}(A[\sqcup_i X_i], A[Y])= \oplus _i A[Mor(X_i,Y)],$$
       where $X_i$ are connected varieties.   
    Denote by $PSh$ the category of contravariant $A$-linear functors $A[Sm]\to Mod(A)$. Consider the pair of adjoint DG quasi-functors
   $$ \Phi^*:  \underrightarrow{A[Sm]}\to   \underrightarrow{A_{tr}[Sm]}$$
   $$\Phi_*:  \underrightarrow{A_{tr}[Sm]} \to  \underrightarrow{A[Sm]}$$
   induced by  $\Phi: A[Sm] \to A_{tr}[Sm]$. 
  The triangulated functor 
  $$\Ho(\Phi_*) : \Ho( \underrightarrow{A_{tr}[Sm]}) \simeq D(PSh_{tr})\to  D(PSh)\simeq \Ho( \underrightarrow{A[Sm]}) $$
  is induced by  the forgetful functor $PSh_{tr} \to  PSh$;  $\Ho(\Phi^*)$ is left adjoint to $\Ho(\Phi_*)$.
  
 Let $I^h\subset  \underrightarrow{A[Sm]} $ (resp. $I_{tr}^h\subset   \underrightarrow{A_{tr}[Sm]} $) be
the DG subcategory generated by objects of the form

(h)   $A[U_{\cdot}] \to A[X]$
 (resp.  $A_{tr}[U_{\cdot}] \to A_{tr}[X]$),
 
where $U_{\cdot} \to X$, $U_i\in Sm_k$, is a hypercovering for the h-topology
(\cite{sv}, \S 10), 
 and let  $I^{h, \Delta} \subset  \underrightarrow{A[Sm]}  $ (resp. $I_{tr}^h\subset   \underrightarrow{A_{tr}[Sm]} $)  be the
subcategory generated by $I^h$ (resp. $I^h_{tr}$) and objects of the form $A[X\times \bA^1] \to A[X]$
(resp. of type $(\Delta)$.)  
The functor $\Phi^*$
takes  $I^h$ into $I_{tr}^h$ and, hence, descends to a DG quasi-functor:
$$\overline \Phi^*:  \underrightarrow{A[Sm]} /I^h \to \underrightarrow{A_{tr}[Sm]} /I_{tr}^h.$$

The following  proposition is essentially a reformulation of a result proven by Voevodsky (\cite{v1}, Theorem 4.1.12).  
\begin{pr}\label{htopologyandtransfers} 
a) The quasi-functor $\overline \Phi^*$ is a homotopy equivalence.\\
b) The composition $\Ho(I_{tr}^h) \to D(PSh_{tr}) \to DM^{\eff}_{\et} (k; A)$ is zero. \\
c) The functor $\overline \Phi^*$ yields  a homotopy equivalences:
$$ \underrightarrow{A[Sm]} /I^{h, \Delta}\iso   \underrightarrow{A_{tr}[Sm]}  /I^{h, \Delta}_{tr} \iso D\cM^{\eff}_{\et} (k; A).$$
\end{pr}
\begin{proof}  We will derive the proposition from Voevodsky's results.
Let $Sh$ be the category of sheaves of $A$-modules on smooth schemes
equipped with the $h$-topology, and let  $A^h[X]$ be the $h$-sheaf associated with  $A[X]$. Then
 \begin{equation}\label{handtr}
A^h[X] \iso A_{tr}[X].
\end{equation}
 This follows by combining (\cite{v2}, Theorem 3.3.5, Proposition 3.3.6) and (\cite{bv}, 2.1.3) with the observation that
  the category of $h$-sheaves on smooth schemes is equivalent to the one on all $k$-schemes of finite type. The last assertion holds because
 every scheme of finite type over a field of characteristic $0$ admits a smooth $h$-cover.

 It follows from (\ref{handtr}) that the forgetful functor $\Ho(\Phi_*): D(PSh_{tr})\to  D(PSh) $ takes every complex in $\Ho(I_{tr}^h)$ to an object of  $\Ho(I^h)$ (\cite{bv}, 1.11). Consider the induced functor:
$$\Ho(\overline \Phi_*) : D(PSh_{tr})/ \Ho(I_{tr}^h) \to  D(PSh)/\Ho( I^h) . $$
We claim that $\Ho(\overline \Phi_*)$ and $\Ho(\overline \Phi^*)$ are inverse one
to the other. Indeed, (\ref{handtr}) implies that $ \Ho(\overline \Phi_*)
\circ \Ho(\overline \Phi^*) \simeq Id$. Next, since $\Ho(\Phi ^*)$ is left adjoint
to $\Ho(\Phi_*)$, $\Ho(\overline \Phi^*)$ is left adjoint to $\Ho(\overline \Phi_*)$.
This implies that $\Ho(\overline \Phi^*)$ is fully faithful. Since  $\Ho(\overline \Phi^*)$ is
obviously surjective, the first part of the proposition is proved.
Part b) follows from  ( \cite{v1}, Theorem 4.1.12).  The last part is a corollary of a) and b).
\end{proof}

\subsection{Compactifications.}\label{comp} Let $A[\overline {Sm}]$ be the category whose objects are triples
$(X, \overline X, j)$, where $\overline X$ is smooth proper scheme
over $\mathbb{C}$ and $j: X\to \overline X $ is an open dense embedding  such
that the complement $\overline X \backslash X$ is a divisor with
normal crossings. The space of morphisms between two such triples
$(X, \overline X, j)$ and $(X', \overline X', j')$  with connected $\overline X$ is freely
generated over $A$ by pairs $(f:X_i\to X', \, g:\overline X_i\to \overline X')$
such that $j' f= g j$.  In general,
  $$Hom_{ A[\overline {Sm}]}(A[ ( \sqcup X_i, \sqcup \overline X_i, j)], \cX')= \oplus_i  Hom_{ A[\overline {Sm}]}(A[(X_i, \overline X_i, j)], \cX'),$$
where $\overline X_i$ are connected and $\cX'=A[(X', \overline X', j')]$.  Let ${\cN}$ be
the DG  subcategory of $\underrightarrow
{A[\overline {Sm}]}$ generated by complexes of the form 

(N)  $A[X, \overline X', j_{X'}] \stackrel{(Id,g)}{\longrightarrow} A[X, \overline X, j_X].$ 

The functor $A[\overline {Sm}]\to A[Sm]$ that takes a triple $(X,
\overline X, j)$ to $X$ induces a  DG quasi-functor 
\begin{equation}\label{resolution}
 \underrightarrow{A[\overline {Sm}]}/{\cN}   \to
 \underrightarrow {A[Sm]}  .
\end{equation}
 \begin{lm}\label{hironaka}
The functor (\ref{resolution}) is a homotopy equivalence. 
\end{lm}
\begin{proof}
Denote by ${\cS} \subset Mor(A[\overline {Sm}])$ the set of
morphisms of the form
\begin{equation}\label{pochti}
(X, \overline X', j_{X'}) \stackrel{(Id,g)}{\longrightarrow} (X,
\overline X, j_{X}).
\end{equation}
It follows from the Hironaka theorem on resolution of singularieties
that the set ${\cS}$ is a left multiplicative system and that, if $A[\overline {Sm}]_{\cS}$ is the
localization of $A[\overline {Sm}]$ by ${\cS}$ (\cite{ks}, \S 7),  the
functor $A[\overline {Sm}]_{\cS}\to  A[Sm]$  induced by $A[\overline {Sm}] \to  A[Sm]$ is
 an equivalence of categories.  The rest of the proof is a bit of abstract  nonsense.  It will suffice to
prove that for any additive category ${\cA}$ and a left
multiplicative system ${\cS} \subset Mor({\cA}) $ the functor
$$T: \Ho(\underrightarrow  {\cA}/{\cN})   \to \Ho(\underrightarrow  {\cA} _{\cS}) $$
is an equivalence of triangulated categories. Here ${\cN}$
denotes the DG  subcategory of $\underrightarrow  {\cA} $ generated by complexes of the form
$X'\stackrel{s}{\to} X$, where $X, X' \in {\cA}$ and $s\in {\cS}$.   
For every $X, Y \in {\cA} $, $i \in \bZ$ we have 
$$Hom _{ \Ho(\underrightarrow  {\cA}/{\cN}) } ( X, Y[i])\simeq Hom _{ \Ho(\underrightarrow  {\cA}^{perf}/{\cN}^{perf} ) } ( X, Y[i])\simeq  \colim \, Hom_{ \Ho(\underrightarrow  {\cA} )} ( C, Y[i]),$$
where the colimit is taken over the cofiltrant category $Q$ of pairs $(C\in  \Ho(\underrightarrow  {\cA}^{perf}), g: C \to X)$ with $cone(g)\in \Ho({\cN}^{perf})$. The subcategory $Q'\subset Q$ formed $(X'\in \cA,
s: X' \to X)$ with $s \in {\cS}$ is cofinal in $Q$. Hence
\begin{equation}\label{ressing}
Hom _{ \Ho(\underrightarrow  {\cA}/{\cN}) } ( X, Y[i])\simeq  \colim \, Hom _{ \Ho(\underrightarrow  {\cA}) } ( X', Y[i])\to Hom _{ \Ho(\underrightarrow  {\cA} _{\cS})} (X, Y[i]).
\end{equation}
(In particular, all groups are $0$ for $i\ne 0$.) 
Since objects of $\cA$ generate   $\Ho(\underrightarrow  {\cA})$,  $\Ho(\underrightarrow  {\cA}/{\cN})$   and are compact in the both categories   (\cite{bv},  Proposition 1.4.2) formula (\ref{ressing}) implies that $T$ is fully faithful.

It remains to check that $T$ is essentially surjective. Indeed, $T$
provides an equivalence  of $\Ho( \underrightarrow  {\cA}/{\cN}  ) $ with  a full triangulated subcategory of $\Ho(\underrightarrow  {\cA} _{\cS}  )  $ which is closed
under under small direct sums and contains ${\cA} _{\cS}  $.
Every subcategory with these properties coincides with  $\Ho(\underrightarrow  {\cA} _{\cS}) $.
\end{proof}

\subsection{DG realizations of $D\cM^{\eff}_{\et} (k; A)$.}  Let ${\cC}$
be a  pretriangulated  cocomplete compactly generated DG category over a commutative ring $A$, and let  $\cF: A[Sm] \to  \cC$ (resp. $\cF:  A[\overline {Sm}] \to \cC$)  be a DG quasi-functor. According to formula  (\ref{dgindcompl}) $\cF$ extends uniquely (up to a unique isomorphism) to a DG quasi-functor
$\underrightarrow{\cF}: \underrightarrow{A[Sm]}\to \cC$ (resp. $\underrightarrow{\cF}: \underrightarrow{A[\overline{Sm}]}\to \cC$)
such that $\Ho(\underrightarrow \cF)$ commutes with arbitrary direct sums. Let $\cT^{h, \Delta} (A[Sm], \cC)$ (resp. $\cT^{h, \Delta, \cN} (A[\overline{Sm}], \cC)$ ) be a full subcategory of $\cT (A[Sm], \cC)$ (resp. $\cT(A[\overline {Sm}] , \cC)$) whose objects are quasi-functors $\cF$ such that $\Ho(\underrightarrow \cF)$ takes complexes of type $(h)$, $(\Delta)$ (resp. $(h)$, $(\Delta)$, $(N)$)  to $0$ in $\Ho(\cC)$. The following result is the basic tool for constructing various
realization functors.
\begin{Th}\label{abrealization}
 The functor $\overline \Phi$ induces an equivalence of triangulated  categories
$$\cT^c(D\cM^{\eff}_{\et} (k; A), \cC) \iso \cT^{h, \Delta} (A[Sm], \cC)\iso \cT^{h, \Delta, \cN} (A[\overline{Sm}], \cC)$$
where $\cT^c(D\cM^{\eff}_{\et} (k; A), \cC)$ is the full subcategory of $\cT(D\cM^{\eff}_{\et} (k; A), \cC)$ formed by quasi-functors $\cF$ such that $\Ho(\cF)$ commutes with arbitrary direct sums.
 \end{Th}
\begin{proof} This is an immediate corollary of part c) of Proposition \ref{htopologyandtransfers}.  \end{proof}

\subsection{Betti realization.}\label{b.r.} $k=\mathbb{C}$.  Let $D_{dg}(Mod(A))$ be  the derived DG
category of modules over a commutative ring $A$ {\it i.e.},  the DG quotient of the category $C(Mod(A))$ of unbounded complexes of $A$-modules by
the subcategory of acyclic ones. 
We apply Theorem \ref{abrealization} to the quasi-functor
         $ C^{sing}_{A}:  A[Sm] \to C(Mod(A)) \to D_{dg}(Mod(A))$
that takes a smooth variety $X$ to its singular chain complex
$C^{sing}(X(\mathbb{C}))\otimes A $.  Since singular homology 
are homotopy invariant and satisfy the descent property for the
$h$-topology   $C^{sing}_{A}$ yields a DG quasi-functor
$$ R_{A}^{Betti}:   D\cM^{\eff}_{\et} (\mathbb{C}; A) \to D_{dg}(Mod(A)) . $$

\subsection{Hodge realization.}\label{h.r.}  Let $\bZ \subset A\subset \bQ$ be a subring,  and let $MHS^A$ be the category of {\it polarizable}
 mixed $A$-Hodge structures (\cite{d3}).
 We shall say that  a Hodge structure $(V_A, W_{\cdot} \subset V_\bQ, F^{\cdot}\subset V_{\mathbb{C}})\in MHS^{A}$  
is {\it effective} if
 \begin{equation}\label{maineffective}
 F^1=0 .
\end{equation}
Note that the
condition (\ref{maineffective}) implies that 
$$W_0=V_\bQ.$$
Denote by $MHS^{A}_{\eff} \subset MHS^{A}$ the full subcategory of effective Hodge structures.
We shall construct
a DG quasi-functor,
\begin{equation}\label{hodgemain}
 R^{Hodge}_A:  D\cM^{\eff}_{\gm, \et} (\mathbb{C};  A)\to  D_{dg}^b(MHS^{A}_{\eff}), 
 \end{equation}
together with an isomorphism of quasi-functors
$$ F \circ R^{Hodge}_A \simeq R_{A}^{Betti}: D\cM^{\eff}_{\gm, \et} (\mathbb{C};  A) \to D_{dg}(Mod(A)).$$
Here $F: D_{dg}^b(MHS^{A}_{\eff}) \to D_{dg}(Mod(A))$ is the forgetful functor.

Beilinson associated with every variety $X$ over $\bC$ an element of the derived category $D^b(MHS^A)$ whose cohomology are Deligne's mixed Hodge structures on (co)homology groups of $X$.
In fact,  his construction gives a DG quasi-functor $  A[\overline{Sm}] \to D_{dg}^b(MHS^{A})$. Let us explain this.
 In  (\cite{bei}, \S 3.9-4.1) Beilinson introduced  auxiliary triangulated categories $K^{b}_{\tilde {\cH}_{p}}$, $D^{b}_{\tilde {\cH}_{ p}}$ of $\tilde p$-Hodge complexes.  We define DG structures on these categories: consider a DG category  $C^{b}_{\tilde {\cH}_{p}}$
 whose objects are $\tilde p$-Hodge complexes and the group  $Hom((\tilde \cF^{\cdot}, d_{\cF}),  (\tilde \cG^{\cdot}, d_{\cG})[n] )$ of degree $n$ morphisms between $\tilde p$-Hodge complexes
 $$(\tilde \cF^{\cdot}, d_{\cF}) = \tilde \cF_A \rar{\alpha_\cF} \tilde\cF'_{\bQ}\lar{\beta_\cF}(\tilde\cF_{\bQ}, \tilde W_{\bQ})\rar{\gamma_\cF} (\tilde\cF'_{\bQ}, \tilde W'_{\bQ})\lar{\delta_\cF} (\tilde\cF_{\bQ}, \tilde W_{\bQ}, \tilde F^{\cdot}),$$
 $$(\tilde \cG^{\cdot}, d_{\cG}) = \tilde \cG_A \rar{\alpha_\cG} \tilde\cG'_{\bQ}\lar{\beta_\cG}(\tilde\cG_{\bQ}, \tilde W_{\bQ})\rar{\gamma_\cG} (\tilde\cG'_{\bQ}, \tilde W'_{\bQ})\lar{\delta_\cG} (\tilde\cG_{\bQ}, \tilde W_{\bQ}, \tilde F^{\cdot})$$
is a subgroup of $\prod Hom(\tilde \cF^{\cdot}, \tilde \cG^{\cdot +n})$
formed by all $h\in \prod Hom(\tilde \cF^{\cdot}, \tilde \cG^{\cdot +n})$ such that \\
a)  $h$ commutes with   $\alpha, \beta,\gamma$ and $ \delta$ (but not necessarily with the differential) \\
b) $h$ preserves the Hodge filtration \\
c) $h$ shifts the weight filtration: $h(\tilde W_{\cdot}) \subset  \tilde W_{\cdot - n}$ \\
d) $d_{\cG}\circ h - (-1)^n h\circ d_{\cF}$ takes  $\tilde W_{\cdot}$ to  $ \tilde W_{\cdot - n -1}$.\\
The formula $d(h):= d_{\cG}\circ h - (-1)^n h\circ d_{\cF}$ defines a differential on $Hom((\tilde \cF^{\cdot}, d_{\cF}),  (\tilde \cG^{\cdot}, d_{\cG})[\cdot ] )$.  
By definition, Beilinson's category $K^{b}_{\tilde {\cH}_{p}}$ is the homotopy category of $C^{b}_{\tilde {\cH}_{p}}$.
We have  a DG functor
\begin{equation}\label{eq2}
C^b( MHS^{A}) \to C^{b}_{\tilde {\cH}_{p}}
\end{equation}
that takes a complex $V^{\cdot}=(V_A, W_{\cdot} \subset V_\bQ, F^{\cdot}\subset V_{\mathbb{C}})$  of Hodge structures to the $\tilde p$-Hodge complex  $\tilde \cF^{\cdot}$ with $\tilde \cF_A=V_A$, $ \tilde\cF'_{\bQ}= \tilde\cF_{\bQ} =V_\bQ$,  $ \tilde\cF'_{\bC}= \tilde\cF_{\bC} =V_\bC$, $\tilde F^{\cdot}=F^{\cdot}$,
and $\tilde W_{ i} (\tilde \cF^{j}_\bQ) = W_{i+j}(V^j_\bQ)$.
  
Denote by $D^{b}_{dg \,  \tilde {\cH}_{p}}$ the DG quotient of  $C^{b}_{\tilde {\cH}_{p}}$ by the subcategory acyclic complexes.\footnote{ Warning: a $\tilde p$-Hodge complex $(\tilde \cF^{\cdot}, d_{\cF})$ is called acyclic if $\tilde \cF_A$ is acyclic.  
The subcomplexes  $\tilde W_{\bQ}$ need not be acyclic.}
Then the functor (\ref{eq2})
descends to a homotopy equivalence  (\cite{bei}, Lemmas 2.1 and 3.11): 
$$D^b_{dg}( MHS^{A}) \iso D^{b}_{dg\, \tilde {\cH}_{p}}.$$
Next, with every $(X, \overline X, j)\in  A[\overline {Sm}] $ Beilinson associated a canonical $\tilde p$-Hodge complex ${\cD}[X, \overline X]$. Set $\tilde R^{Hodge}_A((X, \overline
X, j))= \underline {{\cH}om} ({\cD}[X, \overline X],
A)$. We obtain DG quasi-functors:
$$
\tilde R^{Hodge}_A:  A[\overline {Sm}] \to C^{b}_{\tilde {\cH}_{p}}, 
$$
$$\tilde{{\underrightarrow R}}^{Hodge}_A:   \underrightarrow{A[\overline {Sm}]} \to  \underrightarrow{D^{b}_{dg \,  \tilde {\cH}_{p}}} \iso  \underrightarrow{D^b_{dg}( MHS^{A})} \rar{\cP}  D_{dg}(Ind(MHS^{A})),$$
where $Ind(MHS^{A})$ denotes the ind-completion of $MHS^{A}$ and $\cP$ is induced (via (\ref{dgindcompl}) ) by the embedding $ MHS^{A}\to Ind(MHS^{A})$. By construction, the composition of $\tilde{{\underrightarrow R}}^{Hodge}_A$ with the
forgetful functor    $\underrightarrow{F}:  D_{dg}(Ind(MHS^{A})) \to D_{dg}(Mod(A))$ is isomorphic to $\underrightarrow{A[\overline {Sm}]}\to D\cM^{\eff}_{\et} (\mathbb{C}; A) \rar{ R_{A}^{Betti}} D_{dg}(Mod(A))$. Now $\Ho(\tilde{{\underrightarrow R}}^{Hodge}_A)$ carries 
 complexes of type $(h)$, $(\Delta)$, $(N)$ to $0$, because $Ho(\underrightarrow{F})$ is conservative. Theorem \ref{abrealization} gives 
 $${\underrightarrow R}^{Hodge}_A:   D\cM^{\eff}_{\et} (\mathbb{C}; A)\to D_{dg}(Ind(MHS^{A})).$$
The functor $\Ho({\underrightarrow R}^{Hodge}_A)$ takes every geometric motive to the essential image of the fully faithful embedding $ D^b(MHS^{A})\mono D(Ind(MHS^{A}))$. This yields 
$$R^{Hodge}_A:  D\cM^{\eff}_{\gm, \et} (\mathbb{C};  A)\to  D_{dg}^b(MHS^{A}).$$
We will show in Proposition \ref{Lalb} that the functor $D^b(MHS^{A}_{\eff}) \to D^b( MHS^{A})$ induced by the embedding $MHS^{A}_{\eff} \mono MHS^{A}$ is fully faithful. Since the mixed Hodge structures on the homology groups of any variety are effective,
it follows that the DG quasi-functor  $R^{Hodge}_A$ factors uniquely through $D_{dg}^b(MHS^{A}_{\eff})$. This gives  (\ref{hodgemain}).

If $k\subset \bC$,  we will write  $R^{Hodge}_A$ for the composition
\begin{equation}\label{hodgemaineff}
  D\cM^{\eff}_{\gm, \et} (k;  A) \rar{f^*} D\cM^{\eff}_{\gm, \et} (\mathbb{C};  A)\to  D_{dg}^b(MHS^{A}_{\eff}),
 \end{equation}
where $f^*$ is the base change functor from \S \ref{bcg}.

\subsection{ \'Etale realization.}\label{e.r.}   Fix an algebraic closure $k \subset
\overline k $.   Denote by  $Mod(G, A)$ the category of
discrete $A$-modules\footnote{Recall that a $G$-module $M$ is called discrete if the stabilizer of every element $m\in M $ is open in $G$.}  over the Galois group  $G:= Gal (\overline
k/k)$
 and by $D_{dg}(G,  A)$  its derived DG category.
Let $ C^{\Delta}: D_{dg}(PSh_{tr}^A(Sm_k)) \to D_{dg}(PSh_{tr}^A(Sm_k))$ be the $\bA^1$-homotopy localization endofunctor (\cite{bv}, \S 3.1.1).
Consider a DG quasi-functor
$$\tilde \cS_A:    D_{dg}(PSh_{tr}^A(Sm_k)) \to D_{dg}(G,  A)$$
that takes a complex  of presheaves $F$ to
$$\tilde \cS_A(F)= C^{\Delta}(F)(spec\, k):=  \underset {k \subset k' \subset \overline k} {\colim} C^{\Delta}(F)(spec\, k'),$$
where $k'$ runs through all finite extensions of $k$ in $\overline k$. According to the Suslin-Voevodsky theorem (\cite{sv}) for $X$ one has a canonical isomorphism of $G$-modules
  \begin{equation}\label{susvoetheorem}
  H_i(\tilde \cS_{ \mathbb{Z}/n}(\mathbb{Z}/n_{tr}[X])) \simeq H_i^{\et}(X\times spec\, \overline k, \mathbb{Z}/n).
  \end{equation}
  Here is a more precise statement.
\begin{pr}\label{etalerealization} We have 
 \begin{equation}\label{eq3}
\Ho(\tilde \cS_A) (\Ho(I^{et, \Delta }_{tr}))=0.
 \end{equation}
 Thus, $\tilde \cS_A $ induces a DG quasi-functor
 $$\cS_A:    D\cM^{\eff}_{\et} (k;  A)  \to D_{dg}(G,  A).$$
 The functor $ \cS_{ \mathbb{Z}/n}$  is a homotopy equivalence:
  \begin{equation}\label{proofofsv}
  \Ho(\cS_{ \mathbb{Z}/n}): DM^{\eff}_{et } (k; \mathbb{Z}/n)\iso D(G,  \mathbb{Z}/n).
  \end{equation}
 \end{pr}
  \begin{proof} Assume, first,  $A= \mathbb{Z}/n $. In this case,  formula (\ref{eq3}) is proved in  (\cite{mvw}, Theorem 9.35 ). It is also shown there that $\Ho(\tilde \cS_{\mathbb{Z}/n})$ induces an equivalence of the bounded from above categories:
  $$  DM^{eff }_{-, et } (k; \mathbb{Z}/n)\iso D^-(G,  \mathbb{Z}/n).$$
  This together with (\ref{susvoetheorem}) implies
  $$\Ho(\cS_{ \mathbb{Z}/n}): DM^{\eff}_{et } (k; \mathbb{Z}/n)^{perf}\iso D(G,  \mathbb{Z}/n)^{perf}.$$
 Since $\Ho(\cS_{ \mathbb{Z}/n})$ commutes with arbitrary direct sums the equivalence (\ref{proofofsv}) follows from Corollary \ref{compgen}. 

Let us prove the first claim of the proposition.  It suffices to check that $\Ho(\tilde \cS_{\bZ}) (M)=0$, for every complex  $M\in \Ho(I^{et, \Delta }_{tr})$   of type  $(\Delta)$ and of type $(Et)$ (with $A=\bZ$). In either case $M$ is a complex of torsion-free presheaves.    
 Consider
 the distinguished triangle
    $$ \Ho(\tilde \cS_{\bZ})(M)\otimes \mathbb{Q}/\mathbb{Z}[-1] \to  \Ho(\tilde \cS_{\bZ}) (M) \to \Ho(\tilde \cS_{\bZ}) (M)\otimes \mathbb{Q} \to  \Ho(\tilde \cS_{\bZ}) (M)\otimes \mathbb{Q}/\mathbb{Z}. $$
  The complex  $\Ho(\tilde \cS_{\bZ}) (M)\otimes \mathbb{Q}$ is acyclic by  ( \cite{mvw} , Theorem 9.35 ). The complex  $\Ho(\tilde \cS_{\bZ}) (M)\otimes \mathbb{Q}/\mathbb{Z}$
  is a colimit of acyclic complexes  $\Ho(\tilde \cS_{\bZ})(M)\otimes  \frac{1}{n}\mathbb{Z}/\mathbb{Z}$ over a filtrant  category. This implies that $  \Ho(\tilde \cS_{\bZ}) (M)$
  is acyclic.
\end{proof}
Let $l$ be a prime number.  Denote by $Mod(G, \mathbb{Z}/l^{\cdot})
$ the abelian category whose objects are inverse systems $\cdots \to
N_i \rar{\phi_{i-1}} N_{i-1} \to \cdots  \rar{\phi_1} N_1$,  where $N_i\in Mod(G, \mathbb{Z}/l^{i})$ and $\phi_i$ are
morphisms of  $G$-modules.  A morphism $N_{\cdot} \to N'_{\cdot}$
is a compatible system of homomorphisms $N_i \to N'_i$. Denote and $
D_{dg}(G, \mathbb{Z}/l^{\cdot} )$ the derived DG category of $Mod(G,
\mathbb{Z}/l^{\cdot}) $.  Let $Mod(G, \mathbb{Z} )\to Mod(G,
\mathbb{Z}/l^{\cdot}) $ be the functor that takes $ N\in  Mod(G,
\mathbb{Z}_l )$ to the inverse system $N_{\cdot}: = N\otimes
\mathbb{Z}/l^{\cdot}$, and let
$$ \stackrel{L}{\otimes}   \mathbb{Z}/l^{\cdot}:  D_{dg}(G, \mathbb{Z} ) \to  D_{dg}(G, \mathbb{Z}/l^{\cdot} )$$
be its left derived DG quasi-functor (\cite{dri}, \S 5).  Consider the composition of quasi-functors
\begin{equation}\label{ladicrealization}
R_{\mathbb{Z}/l^{\cdot}}^{\et}: D\cM^{\eff}_{\et} (k;  \mathbb{Z})  \rar{\cS_{\bZ } } D_{dg}(G, \mathbb{Z} ) \rar{\stackrel{L}{\otimes}   \mathbb{Z}/l^{\cdot}}
D(G,  \mathbb{Z}/l^{\cdot}).
\end{equation}
\begin{pr} There is a morphism of DG quasi-functors
  \begin{equation}\label{dtisom}
   \cS_{\mathbb{Z}} \to R_{\mathbb{Z}}^{Betti}:   D\cM^{\eff}_{\et} (\bC;  \mathbb{Z}) \to  D_{dg}(Mod(\bZ))
   \end{equation}
  such that for every integer $n$ the induced morphism
  $$ \cS_{\mathbb{Z}}   \stackrel{L}{\otimes}   \mathbb{Z}/n \to R_{\mathbb{Z}}^{Betti}  \stackrel{L}{\otimes}   \mathbb{Z}/n    $$
  is an isomorphism. In particular,
  \begin{equation}\label{ladiccomp}
  R_{\mathbb{Z}/l^{\cdot}}^{\et}  \simeq R_{\mathbb{Z}}^{Betti}  \stackrel{L}{\otimes}   \mathbb{Z}/l^{\cdot}.
  \end{equation}
\end{pr}
\begin{proof} For the first statement, according to Theorem \ref{abrealization} it suffices to construct a morphism between DG quasi-functors
$$  \tilde \cS_{\mathbb{Z}} \to C^{sing}_\bZ  : \mathbb{Z}[Sm] \to   D_{dg}(Mod(\bZ)). $$
    Consider
  a third functor $ C^{DT}: \mathbb{Z}[Sm] \to  D_{dg}(Mod(\bZ))$ that takes
  a variety $X$ to its Dold-Thom complex $Hom_{top}(\Delta_{top}^{\cdot}, \amalg _d S^d X(\bC))^+$ (\cite{sv}, \S 1).
  We have canonical morphisms 
  \begin{equation}\label{doldthom}
  C^{sing}_\bZ  \to   C^{DT} \leftarrow  \tilde \cS_{\mathbb{Z}} 
  \end{equation}
  with the first arrow being  a quasi-isomorphism (\cite{dt}; \cite{sv},  Theorem 8.2). This yields (\ref{dtisom}).
  
  The second arrow in (\ref{doldthom}) induces a quasi-isomorphism  $  \cS_{\mathbb{Z}}   \stackrel{L}{\otimes}   \mathbb{Z}/n  \to  C^{DT} \stackrel{L}{\otimes}   \mathbb{Z}/n $ by a key result of Suslin and Voevodsky   (\cite{sv},  Theorem 8.3).
\end{proof}
Let  $\overline{Mod}(G, \mathbb{Z}_l) $ be the category of all
$\mathbb{Z}_l$-modules over $G$, and let
$$\lim:  Mod(G, \mathbb{Z}/l^{\cdot} ) \to   \overline{Mod}(G, \mathbb{Z}_l) $$
be the inverse limit functor.
\begin{cor}\label{comst} For every geometric motive $M\in DM^{\eff}_{\gm, \et} (\mathbb{C};  \mathbb{Z})$  and  integer $i$ one has a natural isomorphism
$$\lim H_i (R_{\mathbb{Z}/l^{\cdot}}^{\et}(M))\simeq H_i (R_{\mathbb{Z}}^{Betti}(M)) \otimes \mathbb{Z}_l.$$
\end{cor}
Indeed, for every  geometric motive $M$, the complex
$R_{\mathbb{Z}}^{Betti}(M)$ is perfect, {\it i.e.} quasi-isomorphic
to a finite complex of finitely
 generated free abelian groups. Hence, by (\cite{am}, pp 107-109)
$$\lim H_i (R_{\mathbb{Z}}^{Betti}(M)  \stackrel{L}{\otimes}   \mathbb{Z}/l^{\cdot})\simeq H_i (R_{\mathbb{Z}}^{Betti}(M)) \otimes \mathbb{Z}_l.$$
\subsection{Relation to Huber's approach}\label{h.a.} Huber's Hodge realization functor (\cite{hu1}, Corollary 2.3.5) is the dual to the functor $Ho(R_{\mathbb{Q}}^{Hodge})$ defined above. This can be seen as follows. By Proposition \ref{htopologyandtransfers}
the functor
  \begin{equation}\label{ref.r.}
\Ho(\overline \Phi^*):  \Ho( \underrightarrow{\bQ[Sm]})\to DM^{\eff} (k;  \mathbb{Q})
 \end{equation}
induced by $\bQ[Sm]\to \bQ_{tr}[Sm]$ exhibits $DM^{\eff} (k;  \mathbb{Q})$ as a quotient of  $\Ho( \underrightarrow{\bQ[Sm]}) $.  Thus, giving an isomorphism between two triangulated realization functors on $DM^{\eff} (k;  \mathbb{Q})$ is equivalent to giving an isomorphism
between their compositions with (\ref{ref.r.}). The comparison result follows by inspection.

\begin{rem} To construct the Hodge realization functor,  Huber defined  in (\cite{hu2}) a functor  $\Ho(\overline \Phi_*):  DM^{\eff} (k;  \mathbb{Q})\to DSm_k$, where $DSm_k$ is a certain quotient of the category
$\Ho( \underrightarrow{\bQ[Sm]})$ such that   $\Ho(\overline \Phi_*) \circ \Ho(\overline \Phi^*)$ is the projection.  In the remark on p. 197 of {\it loc. cit.} she conjectured that $\Ho(\overline \Phi_*)$ is an equivalence of categories.  In fact, her conjecture follows from (\cite{sv}, Lemma 5.16).
Note that Huber does not use explicitly  Voevodsky's results on $h$-topology.  The integral theory ({\it i.e.},  $\Ho(R_{\mathbb{Z}}^{Hodge})$) is missing in  (\cite{hu1}, \cite{hu2}).  
\end{rem}
\subsection{1-motives.}\label{1.m.} All the results and constructions in this subsection are borrowed from \cite{bk}, \cite{d3}, \cite{o}.  Let $k$ be a field of characteristic $0$. A 1-motive over $k$ is a complex of group schemes
                        $$M=[\Lambda \rar{u} G],  $$
      where $\Lambda $ is a $k$-lattice ({\it i.e.}, a group scheme such that $\Lambda(\overline k)$ is a discrete $Gal(\overline k/k)$-module isomorphic to $\bZ^r$ as an abelian group  and the canonical morphism
       $\Lambda(\overline k) \times  spec\, \overline k \to \Lambda$ is an isomorphism)   and $G$ is a semi-abelian $k$-scheme.  Morphisms between 1-motives are given by commutative squares.
     The (additive) category of 1-motives is denoted by $\cM_1= \cM_1(k)$. The category $\cM_1\otimes \bQ$ is abelian.
     
     For an abelian group scheme $H$ over $k$ we denote by   $\underline H$ the corresponding \'etale sheaf  on $Sm_k$
 $$\underline{H}(X): = Hom(X, H).$$
If the neutral component $H^0$ is quasi-projective the sheaf
$\underline H $ has a unique structure of a sheaf
with transfers (\cite{bk}, Lemma 1.3.2). Let
     \begin{equation}\label{eqexst}
\tilde {Tot}: C^b(  \cM_1) \to  C^b(Sh^{\et}_{tr})
 \end{equation}
be the DG functor that takes a 1-motive $[\Lambda \rar{u} G] \in \cM_1 \subset  C^b(  \cM_1)$ to the complex 
$$\underline \Lambda \stackrel{u}{\longrightarrow} \underline G.$$
 As a complex,  $\underline \Lambda$ is placed in degree $0$ and $\underline G$ in degree $1$. We introduce a structure of an exact category on $\cM_1$:  a complex  $M^{\cdot}\in C^b(  \cM_1)$
is said to be acyclic if $\tilde {Tot}( M^{\cdot}) $ is acyclic.
The homotopy category of acyclic complexes is strongly generated by short exact sequences. Let $D^b_{dg}(\cM_1)$ be the derived DG category of
$\cM_1$.  
  We get  a DG quasi-functor
\begin{equation}\label{porep}
 D^b_{dg}(\cM_1) \to D^b_{dg} (Sh^{\et}_{tr}) \rar{(\ref{eq111})}  D\cM^{\eff}_{\et} (k;  \mathbb{Z}). 
\end{equation}
Let $d_{\leq 1}D\cM^{\eff}_{\gm, \et} (k;  \mathbb{Z}) $  be the homotopy idempotent completion of the full subcategory of $D\cM^{\eff}_{\gm, \et} (k;  \mathbb{Z})$ strongly generated by motives of smooth curves.  Then, according to (\cite{bk}, Theorem 
2.1.2) the functor  (\ref{porep}) 
induces a homotopy equivalence between  $D^b_{dg}(\cM_1)$ and $d_{\leq 1}D\cM^{\eff}_{\gm, \et} (k;  \mathbb{Z})$. In particular,  the functor  (\ref{porep})  factors through the subcategory of geometric motives:
$$Tot_\bZ: D^b_{dg}(\cM_1) \mono  D\cM^{\eff}_{\gm, \et} (k;  \mathbb{Z}).$$ 
For a subring $A\subset \bQ $ we set 
$$Tot_A:    D^b_{dg}(\cM_1\otimes A)\lar{} D^b_{dg}(\cM_1) \otimes A \rar{Tot}  D\cM^{\eff}_{\gm} (k;  A).$$
The left arrow is a homotopy equivalence (\cite{bk}, Corollary 1.6.2). 
As in the integral case, $Tot_A$  is a homotopy equivalence  between $D^b(\cM_1\otimes A)$ and $d_{\leq 1}D\cM^{\eff}_{\gm, \et} (k;  A) \subset D\cM^{\eff}_{\gm, \et} (k;  A) $.  
For $A=\bQ$ this result was announced by Voevodsky and proved by  Orgogozo  \cite{o}.

According  to (\cite{bk}, \S 5)  the triangulated functor  $\Ho(Tot_\bQ) $ has a left adjoint functor;  thus by Lemma \ref{dgadj} so does the functor $Tot_\bQ$.  We denote the left adjoint DG quasi-functor by
$$LAlb_\bQ:   D\cM^{\eff}_{\gm}(k; \mathbb{Q}) \to   D^b_{dg}({\cM}_1 \otimes \mathbb{Q}).$$

In (\cite{d3}, \S 10.1) Deligne constructed an equivalence of categories 
\begin{equation}\label{deqe}
{\cM}_1(\cC) \iso MHS_1^\bZ.
\end{equation}
where $ MHS_1^\bZ $ is
the full subcategory of  the category $MHS^{\bZ}$ of mixed polarizable Hodge structures 
that consists of torsion-free objects of type $\{(0,0),(0,-1),(-1, 0),(-1,-1)\}$. As a subcategory of an abelian category  $ MHS_1^\bZ $ inherits an exact structure.
It follows easily from (\cite{bk}, Proposition 1.4.1) that (\ref{deqe}) is, in fact, an equivalence of exact categories.  For  $k\subset \mathbb{C}$,  $A\subset \bQ$, we set
$$
 T^{Hodge}_\bZ:  D^b_{dg}({\cM}_1(k))
\to D^b_{dg}({\cM}_1(\cC))\to D^b(MHS_1^\bZ), $$
$$ T^{Hodge}_A:  D^b_{dg}({\cM}_1(k)\otimes A)
\rar{ T^{Hodge}_\bZ \otimes A} D^b(MHS_1^A).
$$
The l-adic realization of 1-motives (\cite{d3}, 10.1) will be denoted by
$$
T_{\mathbb{Z}/l^{\cdot}}^{\et}: D^b_{dg}({\cM}_1(k)) \longrightarrow D^b_{dg}(G, \mathbb{Z}/l^{\cdot} ).
$$

 \section{Proofs}\label{proofs}
 \subsection{Functor $\overline{LAlb_\bQ}$.}\label{lalb}
Fix a subring $\bZ \subset
 A\subset \bQ$. Recall that a mixed  polarizable Hodge structure $(V_A, W_{\cdot} \subset V_\bQ, F^{\cdot}\subset V_{\mathbb{C}}) \in MHS^A $ is called 
 {\it effective} if $F^1=0$.
 We will write
$MHS_{\eff}^A \subset MHS^A $ for the full subcategory
of effective Hodge structures.
Let 
$$\overline {Tot}_A:   D^b_{dg}(MHS_1^A) \to  D^b_{dg}(MHS_{\eff}^A)$$
be the functor induced by the embedding $ MHS_1^A \subset MHS_{\eff}^A$.
The following result was proven in (\cite{bk}, Propositions 15.3.4 and 17.1.1).  Our proof is a variant of the argument given there.
\begin{pr}\label{Lalb}
The functor $\Ho(\overline {Tot}_A)$ is fully faithful.  The functor $\overline {Tot}_\bQ $ has a left
adjoint functor
$$\overline{LAlb}_\bQ:  D^b_{dg}(MHS_{\eff}^\bQ)\to D^b_{dg}(MHS_1^\bQ),$$
 which is {\it t}-exact {\it i.e.}, for every $M\in MHS_{\eff}^\bQ $, $\Ho(\overline {LAlb}_\bQ)(M)$ is isomorphic to an object of
 $MHS_1^\bQ \subset D^b(MHS_1^\bQ).$
\end{pr}
\begin{proof}
  It is known that the category
$MHS^{A}$ of polarizable mixed Hodge structures has
homological dimension $1$ ({\it i.e.},  $Ext^i(M,N)=0$ for every $M,N\in
MHS^{A}$ and every $i>1$. See \cite{bei}, Corollary 1.10). 
The first claim of the
proposition follows
from this fact and the next lemma.
\begin{lm}\label{lm10}
Let ${\cB}$ be an abelian category of homological dimension $\leq
1$ and let ${\cA}\subset {\cB}$ be a full abelian
subcategory
closed under extensions. Then \\
a) homological dimension of ${\cA}$ is at most $1$\\
b) the functor $D^b({\cA})\to D^b({\cB})$ is fully faithful.
\end{lm}
\begin{proof}
a) It is enough to check that for every $M_1, M_2, M_3 \in {\cA}$
the Yoneda product
$$Ext^1_{\cA}(M_1, M_2)\times Ext_{\cA}^1(M_2, M_3) \to Ext_{\cA}^2(M_1, M_3)$$
is $0$. In turn, this is equivalent to showing that for every
extensions
\begin{equation}\label{ext1}
0\to M_2 \to N_1 \to M_1\to 0
\end{equation}
$$0\to M_3 \to N_2 \to M_2\to 0$$
there exists an object $P\in {\cA}$ with a 3 step filtration
$0\subset M_3 \subset N_2 \subset P$ such that the extension
$$0\to N_2/M_3= M_2 \to P/M_3 \to P/N_2 \to 0 $$ is isomorphic to
(\ref{ext1}). Since $Ext^2_{\cB}(M_1,M_3)=0$ there exists $P\in
{\cB}$ with these properties and since ${\cA} \subset {\cB}$ is closed under extensions $P\in {\cA}$.\\
b) It is enough to show that for every $M_1, M_2 \in {\cA}$ and
for every $i\geq 0$ the morphism
\begin{equation}\label{triv1}
Ext_{\cA}^i(M,N) \to Ext_{\cB}^i(M,N)
\end{equation}
is an isomorphism. This is true for $i=0$ because ${\cA}\subset
{\cB}$ is a full subcategory and for $i=1$ because ${\cA}$ is
closed under extensions. If $i > 1$ both groups in (\ref{triv1}) are
trivial by part a).

\end{proof}

 To prove the existence of $\overline {LAlb}_\bQ$  we first
show that the embedding of abelian categories
$$\overline
{tot}: MHS_1^{\mathbb Q} \subset MHS_{\eff}^{ \mathbb{Q}}$$ has a left
adjoint functor
$$\overline {lalb}:  MHS_{\eff}^{ \mathbb{Q}}\to MHS_1^{\mathbb Q}.$$
The functor $\overline {lalb}$ is going to be the composition
$$ MHS_{\eff}^{ \mathbb{Q}}\stackrel{w_{\geq -2}}{\longrightarrow}
 MHS_{\geq -2, eff}^{\mathbb{Q}}\stackrel{\delta}{\longrightarrow}  MHS_1^{\mathbb Q}
 \, , \,
\overline {lalb}: = \delta \circ w_{\geq -2},$$ where $ MHS_{\geq
-2, eff}^{ \mathbb{Q}}$ is a full subcategory of $MHS_{\eff}^{
\mathbb{Q}}$ consisting of mixed Hodge structures that have weights
greater or equal than $-2$, $w_{\geq -2}$ is the functor that takes
a Hodge structure $V$ to the quotient $V/W_{-3}V$. The second
functor $\delta$ is defined as follows. Given  $V\in MHS_{\geq -2,
eff}^{ \mathbb{Q}}$ there exists a unique decomposition of the pure Hodge structure $W_{-2}V $
into the direct sum of a Hodge-Tate substructure and a
substructure $P\subset W_{-2}V $ that has no Hodge-Tate
subquotients.  The existence follows from the semi-simplicity of the category of pure
polarizable Hodge structures.
Set $\delta (V)= V/i(P)$.
  We leave it to the reader to check that that $\delta$ is indeed a
  functor and that the projection $V\to \delta(V)$ extends to
  a morphism of functors $Id \to \delta $.  Composing it with the
  natural morphism $Id \to w_{\geq -2}$ we get
  $$\overline{\nu}: Id \to \overline {tot} \circ \overline {lalb}.$$
Together with the obvious isomorphism
  $$\overline{\mu}: \overline {lalb} \circ \overline {tot} \simeq Id $$
the triple $(\overline {lalb}, \overline{\nu}, \overline{\mu})$ is
an adjunction datum: the compositions
$$ \overline {lalb} \stackrel{\overline {lalb}\, \overline{\nu}}{\longrightarrow}
 \overline {lalb}\circ \overline {tot} \circ \overline {lalb}
 \stackrel{\overline {\mu} \, \overline {lalb}}{\longrightarrow}\overline {lalb}$$
$$ \overline {tot}\stackrel{\overline{\nu}\, \overline {lalb}}{\longrightarrow}
 \overline {tot}\circ \overline {lalb}\circ \overline {tot}
 \stackrel{\overline {lalb}\,\overline {\mu} }{\longrightarrow}\overline {tot}$$
are identity morphisms.
 Furthermore, the functors $w_{\geq -2}$, $\delta $ are exact and so is the composition $\overline{lalb}$.
 The functor $\overline
s$ is also exact. Thus, $(\overline {lalb}, \overline{\nu},
\overline{\mu})$ automatically extends to an adjunction datum
$(\overline {LAlb}_\bQ, \nu, \mu)$ for the derived DG categories. The
$t$-exactness of $\overline {LAlb}_\bQ$ is clear.

\end{proof}

 \subsection{The main theorem.}\label{mr}
  
\begin{Th}\label{onemotives}
  \begin{enumerate}[(a)]
\item{Let $k$ be a field of characteristic $0$ of finite \'etale homological
dimension, $l$  a prime number.  Then, we have 
\begin{equation}\label{etc}
T^{\et}_{\mathbb{Z}/l^{\cdot}} \simeq R_{\mathbb{Z}/l^{\cdot}} ^{\et} \circ Tot_\bZ.\footnote{The DG quasi-functors in this formula are viewed as objects of the triangulated category $\cT(*, *)$ (see \S \ref{dgcat}). This remark also refers to all other isomorphisms below.} 
\end{equation}}
\item{
 Let $k$ be a subfield of  $ \mathbb{C}$ and $A$ a subring of $\bQ$.  Then, we have
\begin{equation}\label{hodgec1}
\overline {Tot}_A \circ T^{Hodge}_A \simeq R^{Hodge}_A \circ Tot_A
\end{equation}
and the morphism of functors
\begin{equation}\label{hodgec2}
\overline{LAlb}_\bQ  \circ R^{Hodge}_\bQ \to T^{Hodge}_\bQ \circ LAlb_\bQ
\end{equation}
induced by $R^{Hodge}_\bQ \to\overline {Tot}_\bQ \circ T^{Hodge}_\bQ\circ LAlb_\bQ \simeq R^{Hodge}_\bQ \circ Tot_\bQ \circ LAlb_\bQ$ is an isomorphism.}
\end{enumerate}
\end{Th}

The rest of this section is devoted to a proof of the theorem.

\subsection{Proof of (\ref{etc}).} We shall show that, for every 1-motive  
$M=[\Lambda \rar{u} G] $, the modules $H^i( R_{\mathbb{Z}/l^{\cdot}} ^{\et} \circ Tot(M)) \in Mod(G, \bZ/l^{\cdot})$ are $0$ for $i\ne 0$  and canonically isomorphic to Deligne's l-adic realization of $M$ for $i=0$. The formula (\ref{etc}) would follow from Theorem \ref{bounded}. 

Since presheaves $\underline \Lambda, \underline G$ are homotopy invariant the canonical morphism
$$C^\Delta(Tot (M)) \to Tot(M)$$ 
is a quasi-isomorphism. Thus
$$R_{\mathbb{Z}/l^{\cdot}}^{\et}\circ Tot(M)\iso (\Lambda(\overline k) \rar{u} G(\overline k)) \stackrel{L}{\otimes}   \mathbb{Z}/l^{\cdot}
\lar{\alpha}  (\Lambda(\overline k) \rar{u +v} G(\overline k)\oplus  \Lambda(\overline k)  \otimes \mathbb{Q} ) \stackrel{L}{\otimes}   \mathbb{Z}/l^{\cdot}$$
$$\iso (G(\overline k)\oplus  \Lambda(\overline k)  \otimes \mathbb{Q}/\bZ ) \stackrel{L}{\otimes}   \mathbb{Z}/l^{\cdot}[-1]\iso
ker(G(\overline k)\oplus  \Lambda(\overline k)  \otimes \mathbb{Q}/\bZ \rar{l^{\cdot}} G(\overline k)\oplus  \Lambda(\overline k)  \otimes \mathbb{Q}/\bZ ).$$
Here $v: \Lambda (\overline k) \to  \Lambda(\overline k) \otimes \mathbb{Q}$ is the canonical embedding, morphism $\alpha$ sending $\Lambda(\overline k)  \otimes \mathbb{Q}$ to zero is a quasi-isomorphism
because $\mathbb{Q}\stackrel{L}{\otimes}   \mathbb{Z}/l^{\cdot}\simeq 0$.
The module at the right-hand side of the formula is  Deligne's l-adic realization of $M$. This completes the proof of (\ref{etc}).

\subsection{Compatibility of $LAlb_\bQ$ with base change.} 
 When proving the remaining part of the theorem, we may assume that $k=\mathbb{C}$.
Indeed,  we have the following general result.
\begin{pr}\label{bchred} Let $k\subset k'$ be any extension,  $A=\bZ$ or $\bQ$, and let
$$
f^*:  D^b_{dg}({\cM}_1(k) \otimes A) \to
D^b_{dg}({\cM}_1(k' )\otimes A),
$$
\begin{equation}\label{pbf}
f^*: D\cM^{\eff}_{\et}(k; A) \to
 D\cM^{\eff}_{\et}(k';  A)
\end{equation}
be the corresponding pullback functors \S \ref{bcg}. Then
\begin{equation}\label{pbfc1}
f^*\circ Tot_A \simeq Tot_A \circ f^*: D^b_{dg}({\cM}_1(k) \otimes A) \to  D\cM^{\eff}_{\gm}(k'; A),
\end{equation}
\begin{equation}\label{pbfc2}
f^*\circ LAlb_\bQ \simeq LAlb_\bQ \circ f^*:  D\cM^{\eff}_{\gm}(k; \mathbb{Q})\to D^b_{dg}({\cM}_1(k' )\otimes \bQ) .
\end{equation}
\end{pr}
\begin{proof} For every abelian group scheme $G$ of finite type over $k$ and for every lattice $\Lambda$ we have canonical isomorphisms
 $$f^{-1}\underline G \simeq \underline {G\times spec\, k'}, \quad
 f^{-1}\underline \Lambda \simeq \underline {\Lambda\times spec\, k'}.$$
Formula (\ref{pbfc1}) now follows from the commutative diagram in \S \ref{bcg}.

The adjunction property of $LAlb_\bQ$, $Tot_\bQ$ together with (\ref{pbfc1}) yield a morphism 
\begin{equation}\label{domoyporamne} 
\alpha: f^*\circ LAlb_\bQ \to LAlb_\bQ \circ f^*.
\end{equation}
  To prove that $\alpha$ is an isomorphism we shall need the following general property of $f^*$.
Recall from (\cite{v1}, Proposition 3.2.8) that for every $E, \in DM^{\eff}_{\gm}(k; A)$,  $ G \in DM^{\eff}(k; A)$  there is inner Hom-object   $\underline{Hom}(E , G) \in  DM^{\eff}(k; A)$. Set $f^*_{tr}=\Ho(f^*)$.
  \begin{lm}\label{ura} The morphism
\begin{equation}\label{domoypora} 
f^*_{tr} (\underline{Hom}(E , G)) \to \underline{Hom}(f^*_{tr}(E) , f^*_{tr}(G))
\end{equation}
defined by the distinguished element of the group
$$  Hom( f^*_{tr}(\underline{Hom}(E , G)) \otimes f^*_{tr}(E) , f^*_{tr}(G))      \leftarrow    Hom( \underline{Hom}(E , G) \otimes E , G). $$
 is an isomorphism.
\end{lm}
\begin{proof}
Recall from (\cite{bv}, \S 4.4) that the projection
$$P: D(PSh_{tr}(Sm_k))\to DM^{\eff}(k; A)$$
has a right adjoint functor $C^M$ that identifies the category
$DM^{\eff}(k; A)$ with the right-orthogonal complement
$$ I^{Zar, \Delta \, \perp}_{tr}(k) \subset D(PSh_{tr}(Sm_k)).$$
We shall first show that the functor
$f^{-1}$ takes $I^{Zar, \Delta \, \perp}_{tr}(k)$ into $I^{Zar, \Delta \, \perp}_{tr}(k') \subset D(PSh_{tr}(Sm_{k'}))$. In fact,
\begin{equation}\label{ezhik}
f^{-1}(I^{Zar \, \perp}_{tr}(k)) \subset I^{Zar \, \perp}_{tr}(k'), \quad  f^{-1}(I^{\Delta \, \perp}_{tr}(k)) \subset I^{\Delta \, \perp}_{tr}(k').
\end{equation}
Let us just check the first inclusion. Let $G\in I^{Zar \, \perp}_{tr}(k)$, $X\in Sm_{k'}$, $U_1\cup U_2 =X$ an open covering of a smooth scheme over $k'$, and let
$$MV(U_1, U_2 ,X): = \mathbb{Z}_{tr}[U_1\cap U_2]\to \mathbb{Z}_{tr}[U_1]\oplus \mathbb{Z}_{tr}[U_2] \to \mathbb{Z}_{tr}[X]$$
be the Mayer-Vietoris complex. We have to show that
 $$Hom_{D(PSh_{tr}(Sm_{k'}))}(MV(U_1,U_2,X), f^*G)=0.$$
There exist $Z\in Sm_k$, an open covering $\tilde U_1\cup \tilde U_2 =Z$, and a $k$-morphism $h: X\to Y$ such that $U_i=h^{-1}(\tilde U_i)$. We have
$$ Hom_{D(PSh_{tr}(Sm_{k'}))}(MV(U_1,U_2,X), f^*G)=   $$
$$ \underset {X \stackrel {g}{\to} Y  \stackrel {g'}{\to} Z } {\colim}
Hom_{D(PSh_{tr}(Sm_{k}))}(MV(g^{\prime -1}(\tilde U_1),g^{\prime -1}(\tilde U_2), Y), G)=0.$$
Here the colimit is taken over the category of triples $(Y\in Sm_k, g,g')$ with $g\circ g'=h$. Proof of the second inclusion in  (\ref{ezhik}) is similar.

As a consequence we see that that the morphism $f^{-1}\circ C^M \to C^M \circ  f^*_{tr}$ induced by $P\circ f^{-1}\circ C^M \to   f^*_{tr}$ is an isomorphism. Let us also observe a natural isomorphism
$$ C^M \underline{Hom}_{DM^{\eff}(k; A)}(P (F) , G) \simeq \underline{Hom}_{D(PSh_{tr}(Sm_k))}(F , C^M (G))$$
coming from the monoidal structure on the functor $P$.

Now we are ready to prove the Lemma. We shall check that  $C^M$ applied to the morphism (\ref{domoypora}) is an isomorphism. Choosing $F\in D(PSh_{tr}(Sm_k))$ with
$P(F)=E$ and using the above remarks we reduce our problem to proving that morphism
$$  f^{-1}\underline{Hom}_{D(PSh_{tr}(Sm_k))}(F , C^M (G)) \to  \underline{Hom}_{D(PSh_{tr}(Sm_{k'}))}(f^{-1} (F) , f^{-1} C^M (G))$$
is an isomorphism\footnote{The proof below shows that $C^M (G)$ can be replaced by an arbitrary complex of presheaves with transfers.} . Moreover, it will suffice to show this for $F= A_{tr}[Z]$, where $Z\in Sm_k$.  For any $X \in Sm_{k'}$,  we have
$$ Hom (A_{tr}[X],  f^{-1}\underline{Hom}(A_{tr}[Z] , C^M (G) ) )
 \simeq $$
$$\underset {X \stackrel {g}{\to} Y } {\colim} \,
Hom( A_{tr}[Y\times Z], C^M (G) ) \simeq
  \underset {X \times _k Z \stackrel {g'}{\to} Y' } {\colim}
Hom ( A_{tr}[Y'], C^M (G) ) $$
$$\simeq
 Hom (A_{tr}[X], \underline{Hom} (f^{-1} ( A_{tr}[Z]  ) , f^{-1} C^M (G))).  $$

\end{proof}

Let us prove that,  for every $M\in  DM^{\eff}_{\gm}(k; \mathbb{Q})$,  morphism $\Ho(\alpha)(M)$ (see (\ref{domoyporamne})) is an isomorphism.  Since the functor $\underline{Hom}(\cdot\, , \mathbb{Q}(1))$ is fully faithful on the subcategory $d_{\leq 1}DM^{\eff}_{\gm}(k'; \mathbb{Q})\subset DM^{\eff}_{\gm}(k'; \mathbb{Q})$   (\cite{bk}, Proposition 4.4.1) it is enough to prove that the morphism
\begin{equation}\label{pbinner}
\underline{Hom}(\Ho(f^* \circ Tot_\bQ \circ LAlb_\bQ)(M) , \mathbb{Q}(1))\to
\underline{Hom}(\Ho(Tot_\bQ \circ LAlb_\bQ \circ f^*) (M) , \mathbb{Q}(1))
\end{equation}
is an isomorphism.
By (\cite{bk}, Cor. 6.2.1), for any geometric effective motive $N$, the morphism
$$\underline{Hom}((Tot_\bQ \circ LAlb_\bQ)(N) , \mathbb{Q}(1))\to
\underline{Hom}( N , \mathbb{Q}(1))$$
induced by $N \to \Ho(Tot_bQ \circ LAlb_\bQ)(N)$ is an isomorphism. Applying this to $N = \Ho(f^*) (M)$ we see that (\ref{pbinner}) is an isomorphism if and only if so is the morphism
\begin{equation}\label{pbinnerr}
\underline{Hom}(\Ho( f^* \circ Tot_\bQ \circ LAlb_\bQ)(M) , \mathbb{Q}(1))\to \underline {Hom}( \Ho(f^*) (M) , \mathbb{Q}(1))
\end{equation}
given by $M \to \Ho(Tot_\bQ \circ LAlb_\bQ)(M)$. Lemma \ref{ura} identifies
 (\ref{pbinnerr}) with the pullback of the isomorphism $$\underline{Hom}(\Ho(Tot_\bQ \circ LAlb_\bQ)(M) , \mathbb{Q}(1))\simeq
\underline{Hom}( M , \mathbb{Q}(1)).$$
\end{proof}

\subsection{Beginning of the proof of (\ref{hodgec1}).}\label{redstep} It will suffice to construct isomorphism (\ref{hodgec1}) for $A=\bZ$. By Proposition \ref{bchred}  we may assume that $k=\bC$.  Corollary \ref{comst} together with formula (\ref{etc}) imply that for every $M\in {\cM}_1(\bC)$
 \begin{equation}\label{vanishing}
H^i (R_{\mathbb{Z}}^{Hodge}\circ Tot_\bZ(M))=0,  \quad i\ne 0.
 \end{equation}
 Thus, by Theorem \ref{bounded} it is enough to construct an isomorphism
 $$ H^0 (R_{\mathbb{Z}}^{Hodge}\circ Tot_\bZ) \simeq H^0(\overline {Tot}_\bZ \circ T^{Hodge}_\bZ): {\cM}_1(\bC)\to MHS_{\eff}^\bZ.$$
Every such isomorphism extends uniquely to (\ref{hodgec1}).

\begin{rem}\label{wrl} Theorem \ref{bounded} is not needed for the weaker result
 $$\Ho(R_{\mathbb{Z}}^{Hodge}\circ Tot_\bZ) \simeq \Ho(\overline {Tot}_\bZ \circ T^{Hodge}_\bZ).$$
For this it suffices to apply the following elementary Lemma.
\begin{lm}(\cite{vol2}, Lemma 8)
 Let $\cE$ and $\cE'$ be abelian categories
and let $\Phi: D^b(\cE) \to D^b(\cE')$ be a triangulated
functor between the corresponding derived categories. 
Assume that, for every object  $X\in \cE$ and $i\ne0$,  we have $H^i(\Phi(X))=0$. 
Assume, in addition, that $\cE$ has homological dimension 1.
 Then there exists a unique isomorphism between $\Phi$
and the derived functor of the  restriction
$H^0(\Phi_{|\cE}): \cE \to \cE'$ extending the identity isomorphism on $\cE$
that commutes with the translation functor.
\end{lm}   
\end{rem}

\subsection{Computation of  $H_0 (R_{\mathbb{Z}}^{Betti}\circ Tot_\bZ)$.}\label{betticompute} 
We shall construct in (\ref{eqkey6})  a functorial isomorphism of abelian groups:
\begin{equation}\label{mains}
\Theta_M:  H^0(R_{\mathbb{Z}}^{Betti} \circ Tot_{\mathbb{Z}}(M)) \simeq   H^0(\overline {Tot}_\bZ \circ T^{Betti}_\bZ(M)) = ker( \Lambda(\bC) \oplus \fg
\stackrel{u\oplus exp}{\longrightarrow} G(\mathbb{C})),
\end{equation}
where $M=[\Lambda \rar{u} G]\in  {\cM}_1(\bC)$,  $\fg$ is the Lie algebra of $G$, and  $exp: \fg \to G(\mathbb{C})$ is the exponential map.
 In the next subsection we check that $\Theta_M$ is compatible with the Hodge structures.

To construct  $\Theta_M$ we consider an auxiliary functor from the category of 1-motives to the category of complexes of abelian groups that takes 
$M$ to the complex
$$C^{sing}_\bZ(M):= cone(\Lambda(\bC) \to N\, Maps(\Delta^{\cdot},
 G(\mathbb{C})))[-1],$$ 
where $N\, Maps(\Delta^{\cdot}, G(\mathbb{C}))$ is the
normalized chain complex of the simplicial abelian group
$Maps(\Delta^{\cdot}, G(\mathbb{C}))$.  Recall that for every abelian Lie group $P$
the complex $N\, Maps(\Delta^{\cdot},
P)$ computes the homotopy groups of $P$. In particular,
$$
H^i(N\, Maps(\Delta^{\cdot}, G(\mathbb{C})))=
\begin{cases}
0 & \text{if $i\ne -1 $}\\
\pi_1(G(\mathbb{C})) & \text{otherwise.}
\end{cases}
$$
We construct a functorial quasi-isomorphism 
$$\phi: C^{sing}_\bZ(M) \rar{}  cone(\Lambda(\bC) \oplus \fg \to  G(\mathbb{C}))[-1] $$
as follows.
   $$
\def\normalbaselines{\baselineskip20pt
\lineskip3pt  \lineskiplimit3pt}
\def\mapright#1{\smash{
\mathop{\to}\limits^{#1}}}
\def\mapdown#1{\Big\downarrow\rlap
{$\vcenter{\hbox{$\scriptstyle#1$}}$}}
\begin{matrix}
   \to & N_2 Maps(\Delta^{\cdot}, G(\mathbb{C}))  &   \stackrel{d_1}{\to} &
   \Lambda \oplus
 N_1 Maps(\Delta^{\cdot}, G(\mathbb{C})) & \stackrel{d_0}{\to} & N_0 Maps(\Delta^{\cdot}, G(\mathbb{C})) \cr
  & \mapdown{0}  & & \mapdown{\phi_0}& & \mapdown{Id}\cr
  \to &  0  & \to &  \Lambda(\bC) \oplus \fg&
\to & G(\mathbb{C}) .
\end{matrix}
$$
 Here the map $\phi_0 $ is the sum of the identity map on $\Lambda(\bC)$ and
the homomorphism
\begin{equation}\label{dkc}
 N_1\, Maps(\Delta^{\cdot}, G(\mathbb{C})) \to
\fg
\end{equation}
 defined as follows. An element of $ N_1
Maps(\Delta^{\cdot}, G(\mathbb{C}))$ is a continuous map $\gamma:
\Delta ^1 =[0,1] \to G(\mathbb{C})$ such that $\gamma(0)=0$. Let
$\tilde \gamma: [0,1] \to \fg$ be the lifting of $\gamma$ such
that $\tilde \gamma(0)=0$. The map (\ref{dkc}) takes $\gamma $ to
$\tilde \gamma(1)$.
\begin{pr}\label{opyat}
 The morphism $\phi $ is a quasi-isomorphism.
\end{pr}
\begin{cor} There is a functorial isomorphism
$$H^0(\phi):  H^0(C^{sing}_\bZ(M))\iso  H^0(\overline {Tot}_\bZ \circ T^{Betti}_\bZ (M)).$$
\end{cor}
\begin{proof}[Proof of \ref{opyat}]
First, we have to show that $\phi$ is a morphism of complexes.\\
({\it i}) $\phi_0 \, d_1=0$. Indeed, for every $\nu \in  N_1
Maps(\Delta^{\cdot}, G(\mathbb{C}))$ the map $d_1(\nu): [0,1]\to
G(\mathbb{C})$ takes the boundary points $0$ and $1$ to $0$.
Moreover, the induced map $S^1\to G(\mathbb{C})$ is contractible. It
follows that
$\widetilde {d_1 (\nu)}(1)=0$.\\
({\it ii}) $ d_0 =  (u\oplus exp)\, \phi_0$. Indeed, for every
$(\lambda \oplus \gamma)\in  \Lambda \oplus
 N_1\, Maps(\Delta^{\cdot}, G(\mathbb{C}))$
$$ d_0 (\lambda \oplus \gamma)= u(\lambda )\cdot \gamma(1)=
(u\oplus exp)\, \phi_0(\lambda \oplus \gamma).$$ To show that $\phi$
is a quasi-isomorphism consider the morphism
$$N\,Maps(\Delta^{\cdot}, \fg) \stackrel{exp_*}{\longrightarrow}
N\,Maps(\Delta^{\cdot}, G(\mathbb{C} ))$$ 
induced by
the homomorphism $exp:  \fg \to G(\mathbb{C})$. Let
$$ N Maps(\Delta^{\cdot},  \fg) \stackrel{s}{\to} C^{sing}_\bZ(M)$$
be the composition of $exp_*$ and the morphism
$N\,Maps(\Delta^{\cdot}, G(\mathbb{C} ))\to C^{sing}_\bZ(M)$. It is easy
to see that $\phi$ factors through the morphism
$$\hat \phi : cone(s) \to T^{Betti}_{\mathbb{Z}}(M).$$
Observe that the map
$$N_i \,Maps(\Delta^{\cdot},  \fg) \stackrel{exp}{\to} N_i\, Maps(\Delta^{\cdot}, G(\mathbb{C}))$$
is an isomorphism for every $i>0$ (for every continuous map
$\Delta^i \to G(\mathbb{C})$ sending all the faces but one to $0$ lifts uniquely to a map $\Delta^i \to \fg$ with the same
property). It follows that $\hat \phi$ is a quasi-isomorphism. Since
the complex $N Maps(\Delta^{\cdot}, \fg)$ is acyclic (for its
cohomology groups are the homotopy groups of $ \fg$), $\phi$ is
a quasi-isomorphism as well.
 \end{proof}
It remains to construct a functorial quasi-isomorphism
\begin{equation}\label{mainequation}
H^0(R_{\mathbb{Z}}^{Betti}\circ Tot_{\mathbb{Z}})(M) \iso H^0(C^{sing}_\bZ(M)).
\end{equation}
With rational coefficients (\ref{mainequation}) can be easily derived from  the Eilenberg-MacLane cube construction. We explain this short proof in \S \ref{e.m.}. 
The integral statement is more involved.
Consider the double complex of presheaves with transfers
    $$
\def\normalbaselines{\baselineskip20pt
\lineskip3pt  \lineskiplimit3pt}
\def\mapright#1{\smash{
\mathop{\to}\limits^{#1}}}
\def\mapdown#1{\Big\downarrow\rlap
{$\vcenter{\hbox{$\scriptstyle#1$}}$}}
\begin{matrix}
  \mathbb{Z}_{tr}[\Lambda \times \Lambda] &    \stackrel{p_1 + p_2 - m }{\longrightarrow} &
\mathbb{Z}_{tr}[\Lambda]  \cr
 \mapdown{u\times u } & &\mapdown{u} \cr
 \mathbb{Z}_{tr}[G\times G] &    \stackrel{p_1 + p_2 - m  }{\longrightarrow}   &
\mathbb{Z}_{tr}[G],
\end{matrix}
$$
where $p_i, m : \mathbb{Z}_{tr}[\Lambda \times
\Lambda] \to \mathbb{Z}_{tr}[\Lambda]$  (resp. $p_i, m :
\mathbb{Z}_{tr}[G \times G] \to \mathbb{Z}_{tr}[G]$) are the maps
induced by the projections and the addition operation on $\Lambda$
(resp. $G$). Denote by $\tilde M$ the associated total complex
shifted so that $\mathbb{Z}_{tr}[G]$ is in cohomological degree $1$.
We shall use the same notation $\tilde M $ for the corresponding
motive. The construction of $\tilde M$ is functorial: sending $M$ to
$\tilde M$ we get a (nonadditive) functor from the category of
1-motives to the triangulated category of \'etale Voevodsky motives.  Next, we
have a functorial morphism:
$$ \tilde M \to Tot_{\mathbb{Z}}(M),$$
 induced by the map of double complexes:
    $$
\def\normalbaselines{\baselineskip20pt
\lineskip3pt  \lineskiplimit3pt}
\def\mapright#1{\smash{
\mathop{\to}\limits^{#1}}}
\def\mapdown#1{\Big\downarrow\rlap
{$\vcenter{\hbox{$\scriptstyle#1$}}$}}
\begin{matrix}
  \mathbb{Z}_{tr}[\Lambda \times \Lambda] &    \to &
\mathbb{Z}_{tr}[\Lambda] && \longrightarrow && \underline \Lambda \cr
 \mapdown{u\times u } & &\mapdown{u} &&&&  \mapdown{u} \cr
 \mathbb{Z}_{tr}[G\times G] &    \to & \mathbb{Z}_{tr}[G] && \longrightarrow
 && \underline G.
\end{matrix}
$$
Consider the induced map of Betti realizations
\begin{equation}\label{ax11}
 \Ho(R_{\mathbb{Z}}^{Betti})( \tilde M )\to \Ho(R_{\mathbb{Z}}^{Betti}\circ
Tot_{\mathbb{Z}})(M).
\end{equation} 
  Since the complex
$R_{\mathbb{Z}}^{Betti}\circ Tot_{\mathbb{Z}}(M)$ has nontrivial cohomology only in degree
$0$ (see (\ref{vanishing}))  we get from (\ref{ax11})
$$\psi: \tau_{\geq 0} \Ho(R_{\mathbb{Z}}^{Betti})( \tilde M
)\to  \tau_{\geq 0} \Ho(R_{\mathbb{Z}}^{Betti}\circ {Tot}_{\mathbb{Z}})(M)\simeq \Ho(R_{\mathbb{Z}}^{Betti}\circ {Tot}_{\mathbb{Z}})(M)$$
\begin{pr}\label{wer}
 The morphism $\psi $ is a quasi-isomorphism.
\end{pr}
\begin{cor} We have a functorial isomorphism
$$H^0(\psi):  H^0(R_{\mathbb{Z}}^{Betti})( \tilde M)\iso H^0(R_{\mathbb{Z}}^{Betti} \circ
Tot_{\mathbb{Z}})(M).$$
\end{cor}
\begin{proof}[Proof of \ref{wer}]  We shall first prove the proposition in the following two
special cases. \\
({\it i})\,  $G=0$. In this case the proposition is equivalent to the
exactness of the sequence
$$\mathbb{Z}[\Lambda(\bC) \times \Lambda(\bC)]   \to \mathbb{Z}[\Lambda(\bC)] \to  \Lambda(\bC) \to 0 .$$
({\it ii})\, $\Lambda =0$. The exact triangle
$$\mathbb{Z}_{tr}[G\times G][-1]     \to  \mathbb{Z}_{tr}[G][-1] \to \tilde
M \to \mathbb{Z}_{tr}[G\times G]$$ yields a long exact sequence
  \begin{equation}\label{kazh} \to H_1(G(\mathbb{C})\times G(\mathbb{C})) \stackrel{p_{1* }+ p_{2*} - m_*  }{\longrightarrow}
  H_1(G(\mathbb{C}))\stackrel{\alpha  }{\longrightarrow} H^0(R_{\mathbb{Z}}^{Betti}( \tilde M
))\to 
\end{equation}
$$H_0(G(\mathbb{C})\times G(\mathbb{C})) \stackrel{p_{1* }+ p_{2*} - m_*  }{\longrightarrow}
  H_0(G(\mathbb{C}))\to H^1(R_{\mathbb{Z}}^{Betti}( \tilde M))\to 0.$$
The map $p_{1* }+ p_{2*} - m_* $ is $0$
on $ H_1$ and an isomorphism on $H_0$.
  It follows that the cohomology groups of $R_{\mathbb{Z}}^{Betti}( \tilde
  M)$ are trivial in positive degrees and identified (via the map $\alpha$) with
  $H_1(G(\mathbb{C}))$ in degree $0$. 
  The following result completes
  the proof of the proposition for type ({\it ii}) 1-motives.
\begin{lm}\label{roitman} The morphism $i: \mathbb{Z}_{tr}[G] \to  \underline G$
induces a quasi-isomorphism \begin{equation}\label{keycomp}
H_1(G(\mathbb{C}))[1] \simeq R_{\mathbb{Z}}^{Betti} (\underline G)
\end{equation}
\end{lm}
 \begin{proof} By Corollary \ref{comst} it suffices to show that, for every prime $l$, $i$ induces a quasi-isomorphism
$$H_1(R_{\mathbb{Z}/l^{\cdot}}^{\et}(\bZ_{tr} [G]))  [1] \simeq R_{\mathbb{Z}/l^{\cdot}}^{\et}(\underline G). $$ 
 By
definition of l-adic realization (\ref{ladicrealization})
this amounts to showing that the map
$$H_1(C^{\Delta}( \bZ_{tr} [G])(\bC) \stackrel{}{\otimes} \mathbb{Z}/l^n) [1] \to G(\mathbb{C})\stackrel{L}{\otimes} \mathbb{Z}/l^n\simeq
 G(\mathbb{C})_{l^n} [1],$$
 where $G(\mathbb{C})_{l^n}$ denotes the group of $l^n$-torsion points of $G(\mathbb{C})$, is a
quasi-isomorphism. 
We have an exact sequence
$$H_1(C^{\Delta}( \bZ_{tr} [G])(\bC) )\otimes \bZ/l^n  \mono H_1(C^{\Delta}( \bZ_{tr} [G])(\bC) \stackrel{}{\otimes} \mathbb{Z}/l^n) \epi H_0(C^{\Delta}( \bZ_{tr} [G])(\bC) )_{l^n}.$$
  Using part (b) of the generalized
Roitman theorem (\cite{bk}, Theorem 14.4.5) we find  that $H_1(C^{\Delta}( \bZ_{tr} [G])(\bC) )\otimes \bZ/l^n =0$. 
Hence,
$H_1(C^{\Delta}( \bZ_{tr} [G])(\bC) \stackrel{}{\otimes} \mathbb{Z}/l^n) $ is isomorphic to  $H_1(C^{\Delta}( \bZ_{tr} [G])(\bC) )_{l^n}$ which maps isomorphically via the Albanese morphism
$$H_0(C^{\Delta}( \bZ_{tr} [G])(\bC) )\to G(\bC)$$ 
to $G(\bC)_{l^n}$ (\cite{bk}, Theorem 14.4.5 (a)).
\end{proof}
To prove the proposition in general, observe that every 1-motive $M$ is an
extension of a 1-motive of type ({\it i}) by a 1-motive of type 
({\it ii}):
$$0\to M_1 \to M \to M_2 \to 0 \quad  $$
and for extensions of this special form the sequence
$$ \tilde M_1 \to \tilde M \to \tilde M_2  $$
can be completed to a distinguished  triangle (since as a complex of presheaves  $\tilde M_1$ is a subcomplex of $\tilde M $ and $\tilde M_2$ is the quotient of 
$\tilde M $ modulo $\tilde M_1$). Hence, in the commutative diagram
   \begin{equation}\label{lemext}
\def\normalbaselines{\baselineskip20pt
\lineskip3pt  \lineskiplimit3pt}
\def\mapright#1{\smash{
\mathop{\to}\limits^{#1}}}
\def\mapdown#1{\Big\downarrow\rlap
{$\vcenter{\hbox{$\scriptstyle#1$}}$}}
\begin{matrix}
   \Ho(R_{\mathbb{Z}}^{Betti}( \tilde M_1 ))&   \to &
\Ho(R_{\mathbb{Z}}^{Betti}( \tilde M ))&\to &  \Ho( R_{\mathbb{Z}}^{Betti}(
\tilde M_2 ))\cr
 \mapdown{ } & &\mapdown{\tilde \psi} &&  \mapdown{} \cr
 \Ho(R_{\mathbb{Z}}^{Betti}\circ Tot_{\mathbb{Z}}) ( M_1 )&   \to
 & \Ho(R_{\mathbb{Z}}^{Betti}\circ Tot_{\mathbb{Z}}(M ))&\to &
\Ho(R_{\mathbb{Z}}^{Betti}\circ Tot_{\mathbb{Z}})( M_2 )
\end{matrix}
\end{equation}
both rows are distinguished triangles. Since the complexes $
R_{\mathbb{Z}}^{Betti}( \tilde M_i )$ are acyclic in positive
degrees $ R_{\mathbb{Z}}^{Betti}( \tilde M )$ is also acyclic in positive degrees. Thus, it
suffices to show that $\tilde \psi$ induces an isomorphism on $H^0$.
The diagram (\ref{lemext}) yields a diagram
$$
\def\normalbaselines{\baselineskip20pt
\lineskip3pt  \lineskiplimit3pt}
\def\mapright#1{\smash{
\mathop{\to}\limits^{#1}}}
\def\mapdown#1{\Big\downarrow\rlap
{$\vcenter{\hbox{$\scriptstyle#1$}}$}}
\begin{matrix}
 H^0(R_{\mathbb{Z}}^{Betti}( \tilde M_1 ))&   \to &
H^0(R_{\mathbb{Z}}^{Betti}( \tilde M ))&\to &
H^0(R_{\mathbb{Z}}^{Betti}( \tilde M_2 )) &\to  0\cr \mapdown{\psi_1
} & &\mapdown{\psi} &&  \mapdown{\psi_2} \cr
H^0(R_{\mathbb{Z}}^{Betti}\circ Tot_{\mathbb{Z}} ( M_1 ))&
\stackrel{\beta}{\hookrightarrow }& H^0(R_{\mathbb{Z}}^{Betti}\circ
Tot_{\mathbb{Z}}(M ))&\to & H^0(R_{\mathbb{Z}}^{Betti}\circ
Tot_{\mathbb{Z}}( M_2 ))& \to 0
\end{matrix}
$$
with exact rows (and injective  map $\beta $). Since the $\psi_i$ are
isomorphisms  $\psi$ is an isomorphism as well.
\end{proof}
Each term of the complex $\tilde M$ is a direct sum of representable
presheaves. It follows, that
 the Betti realization of $ \tilde M $ is canonically isomorphic to the total complex
$C^{sing}_\bZ(\tilde M)$  of the following double complex
    $$
\def\normalbaselines{\baselineskip20pt
\lineskip3pt  \lineskiplimit3pt}
\def\mapright#1{\smash{by
\mathop{\to}\limits^{#1}}}
\def\mapdown#1{\Big\downarrow\rlap
{$\vcenter{\hbox{$\scriptstyle#1$}}$}}
\begin{matrix}
  \mathbb{Z}[\Lambda(\bC) \times \Lambda(\bC)] &    \stackrel{p_1 + p_2 - m }{\longrightarrow} &
\mathbb{Z}[\Lambda(\bC)]  \cr
 \mapdown{u\times u } & &\mapdown{u} \cr
 C^{sing}_\bZ(G(\mathbb{C})\times G(\mathbb{C})) &    \stackrel{p_1 + p_2 - m  }{\longrightarrow}   &
C^{sing}_\bZ(G(\mathbb{C})),
\end{matrix}
$$
  where
$C^{sing}_\bZ(\cdot)$ denotes the normalized singular chain complex of a topological
space {\it i.e.}, the normalized complex of the simplicial abelian group $\bZ[Maps(\Delta^{\cdot}, \cdot)]$. Define a morphism
\begin{equation}\label{secmap}
R_{\mathbb{Z}}^{Betti}( \tilde M )= C^{sing}_\bZ(\tilde M)  \rightarrow
C^{sing}_\bZ(M) 
\end{equation}
$$
\def\normalbaselines{\baselineskip20pt
\lineskip3pt  \lineskiplimit3pt}
\def\mapright#1{\smash{
\mathop{\to}\limits^{#1}}}
\def\mapdown#1{\Big\downarrow\rlap
{$\vcenter{\hbox{$\scriptstyle#1$}}$}}
\begin{matrix}
  \mathbb{Z}[\Lambda(\bC) \times \Lambda(\bC)] &    \to &
\mathbb{Z}[\Lambda(\bC)] && \longrightarrow && \Lambda(\bC) \cr
 \mapdown{u\times u } & &\mapdown{u} &&&& \mapdown{u}\cr
 C^{sing}_\bZ(G(\mathbb{C})\times G(\mathbb{C})) &    \to   &
C^{sing}_\bZ(G(\mathbb{C}))&&\longrightarrow &&  N\,
Maps(\Delta^{\cdot}, G(\mathbb{C})).
\end{matrix}
$$
Here the homomorphism $\mathbb{Z}[\Lambda(\bC)] \to \Lambda(\bC)$ takes a formal combination of elements of $\Lambda(\bC)$ the their sum in $\Lambda(\bC)$ and
the homomorphism of complexes $C^{sing}_\bZ(G(\mathbb{C}))\to  N\,
Maps(\Delta^{\cdot}, G(\mathbb{C}))$ is induced by the homomorphism of simplicial abelian groups $\bZ[Maps(\Delta^{\cdot}, G(\mathbb{C}))]\to Maps(\Delta^{\cdot}, G(\mathbb{C}))$.
Since by Proposition \ref{opyat} the complex $C^{sing}_\bZ(M)$ has nontrivial cohomology only in degree $0$ we get from  (\ref{secmap}) 
$$\rho:  \tau_{\geq 0}C^{sing}_\bZ(\tilde M)  \to
 \tau_{\geq 0} C^{sing}_\bZ(M)\simeq C^{sing}_\bZ(M).$$
\begin{pr}\label{pr39}
 The morphism $\rho $ is a quasi-isomorphism.
\end{pr}
\begin{cor} There is  functorial isomorphism
$$H^0(\rho) \circ H^0(\psi)^{-1}:  H^0(R_{\mathbb{Z}}^{Betti}\circ Tot_{\mathbb{Z}})(M) \iso H^0(C^{sing}_\bZ(M)). $$
\end{cor}
\begin{proof}[Proof of \ref{pr39}] We know from  Propositions
\ref{wer} and \ref{opyat} that the complexes $\tau_{\geq 0}C^{sing}_\bZ(\tilde M)$, $C^{sing}_\bZ(M)$ have nontrivial cohomology only in degree $0$ and that the functors
$$M\mapsto  H^0(C^{sing}_\bZ(\tilde M))$$
$$ M\mapsto H^0(C^{sing}_\bZ(M)) $$
  from the category of 1-motives to the category of abelian groups are exact.
  Thus, it suffices to check the proposition for 
1-motives of the form
$[\Lambda \rar{} 0]$ and for 1-motives of the form  $[0\rar{} G]$. In the first case the statement is trivial. If $M= [0\rar{} G]$
then by (\ref{kazh}) the map $\bZ_{tr}[G][-1] \to  \tilde M$ induces an isomorphism $H_1(G(\mathbb{C}))\iso H^0(C^{sing}_\bZ(\tilde M))$. Hence the following lemma completes the proof.
\end{proof}
\begin{lm} The morphism $C^{sing}_\bZ(G(\mathbb{C}))\to N\,
Maps(\Delta^{\cdot}, G(\mathbb{C}))$ induces a quasi-isomorphism
$$
H_1(G(\mathbb{C}))[1] \simeq N\, Maps(\Delta^{\cdot}, G(\mathbb{C})).
$$
\end{lm}
\begin{proof} The map
$$H_1(G(\mathbb{C}))[1] \to N\, Maps(\Delta^{\cdot},
G(\mathbb{C}))\simeq \pi_1(G(\mathbb{C}))[1]$$ is the left inverse
to the Hurewicz isomorphism $\pi_1(G(\mathbb{C}))[1]\to
H_1(G(\mathbb{C}))[1]. $
\end{proof}

We define $\Theta_M$ to be the composition of $3$ isomorphisms
\begin{equation}\label{eqkey6}
H^0(R_{\mathbb{Z}}^{Betti} \circ Tot_{\mathbb{Z}}(M))  \rar{H^0(\psi)^{-1}} H^0(R_{\mathbb{Z}}^{Betti}( \tilde M))\rar{ H^0(\rho)} H^0(C^{sing}_\bZ(M))
\end{equation}
$$\rar{H^0(\phi)} H^0(\overline {Tot}_\bZ \circ T^{Betti}_\bZ (M)).$$
\subsection{Remark:  $Tot_\bQ$ via additivization.}\label{e.m.}  
The complex defining $\tilde M$ is a truncation of  the Eilenberg-MacLane cube construction.  In fact, one can use the whole  Eilenberg-MacLane
 complex to give a short conceptual proof of (\ref{mainequation}) with rational coefficients. 
 
   Let $Fct(Mod(\bZ), Mod(\bZ))$ be the category of all (not necessary additive) functors $F:Mod(\bZ)\to Mod(\bZ)$. Recall from (\cite{lp}, \S 13.2.2; 13.2.6) that  the Eilenberg-MacLane cube complex 
  $$\cdots \to Q^\bZ_{2} \to Q^\bZ_{1}\to Q^\bZ_{0}$$ 
 is a complex over $Fct(Mod(\bZ), Mod(\bZ))$ whose terms are functors of the form $F(V)= \bZ[V^{2^n}]$, where $\bZ[V^{2^n}]$ denotes the free abelian group generated by the set $V^{2^n}$, and whose $i$-th homology functor $H_i(  Q^\bZ_{\cdot}) $ is isomorphic to the stable homology
of the Eilenberg-MacLane spaces:
\begin{equation}\label{stable}
    H_i(  Q^\bZ_{\cdot})(V)= H_{i+n}(K(V, n)), \quad n\geq i+1.
   \end{equation}  
 In particular,  for positive $i$,  $H_i(  Q^\bZ_{\cdot})(V)$ is a  torsion group and $H_0(  Q^\bZ_{\cdot})(V)$ is canonically isomorphic to $V$.  It follows that the complex  $ Q^\bQ_{\cdot}=  Q^\bZ_{\cdot} \otimes \bQ$ is a resolution of the functor $F(V)=V\otimes \bQ$.
 Let $[\Lambda \rar{u} G] \in \cM_1 \subset  C^b(  \cM_1)$ be a 1-motive.  The complex of representable presheaves $Q^\bQ(M)$
 $$Q^\bQ(M)(X):= cone(Q^\bQ_{\cdot}(\underline \Lambda(X)) \stackrel{u}{\longrightarrow}Q^\bQ_{\cdot}(\underline G(X)))[-1], \quad  X\in Sm$$
 is a resolution of the complex $cone(\underline \Lambda \stackrel{u}{\longrightarrow} \underline G)[-1] \otimes \bQ$.
 Consider the sheafification  $Q^{\bQ, h}(M)$ of $Q^\bQ(M)$ for the h-topology on $Sm$. The sheafification of a representable presheaf $\bQ[Y]$ is canonically isomorphic to  $\bQ_{tr}[Y]$ (\cite{v2}, Theorem 3.3.5, Proposition 3.3.6); $\underline \Lambda$, $ \underline G$ are already sheaves (\cite{v2}, Theorem 3.2.9).
 Thus, by Proposition \ref{htopologyandtransfers} 
$$ Q^{\bQ, h}(M)\simeq  Tot_\bQ(M)$$
  in $DM^{\eff}(\bC, \mathbb{Q})$. On the other hand, as every term of  $Q^{\bQ}(M)$ is a representable presheaf we can compute $R_{\mathbb{Q}}^{Betti} \circ  Q^{\bQ, h}(M)$ by applying $C^{sing}_\bQ$ to $Q^{\bQ}(M)$ termwise. The resulted complex is isomorphic to the total complex of the simplicial complex
   $$ cone(Q^\bQ_{\cdot}( Maps(\Delta^{\cdot}, \Lambda(\bC))) \to  Q^\bQ_{\cdot}( Maps(\Delta^{\cdot}, G(\mathbb{C}) )))[-1] $$
  Applying (\ref{stable}) again we find that  $R_{\mathbb{Q}}^{Betti} \circ  Q^{\bQ, h}(M)$ is canonically quasi-isomorphic to  $ C^{sing}_\bQ(M)$. Summarizing,  we find
  $$R_{\mathbb{Q}}^{Betti} \circ Tot_\bQ(M) \simeq C^{sing}_\bQ(M).$$
  The above argument expresses the idea that the functor $Tot_\bQ$ is the additivization of the functor
that takes a 1-motive $M$
to the motive $cone(\bQ_{tr}[\Lambda] \to \bQ_{tr}[G])[-1]$. We refer the reader to (\cite{k}) for the definition of additivization.

 \subsection{Computation of $H^0(R_{\mathbb{Z}}^{Hodge} \circ Tot_{\mathbb{Z}}(M)) $.}\label{c.h.r.}  
  The abelian groups $H^0(R_{\mathbb{Z}}^{Betti} \circ Tot_{\mathbb{Z}}(M)) $,   $H^0(\overline {Tot}_\bZ \circ T^{Betti}_\bZ(M)) $  underlie the Hodge
structures $H^0(R_{\mathbb{Z}}^{Hodge} \circ Tot_{\mathbb{Z}}(M)) $,   $H^0(\overline {Tot}_\bZ \circ T^{Hodge}_\bZ(M)) $ respectfully. 
Let us show that the isomorphism
(\ref{mains}) is compatible with the Hodge structures.  It will suffice to check that   $\Theta_{M}^{-1}$ preserves the
weight and Hodge filtrations. 

The weight filtration on
$H^0(\overline {Tot}_\bQ \circ T^{Hodge}_\bZ(M)) $ is induced by a filtration on $M=[\Lambda \rar{u} G]$ by submotives:
if $G$ is an extension of an abelian variety $A$ by a torus $T$,
$W_{-2}M = (0, T) \subset W_{-1}M = (0, G) \subset W_0 M= M$. 
The Hodge structure $H^0(R^{Hodge}_\bZ (\underline T)[-1])$ is
pure of weight $-2$ and the Hodge structure $H^0(R^{Hodge}_\bZ
(\underline G)[-1])$ has weights $-2, -1$\footnote{For the morphism
$\mathbb{Z}_{tr}[G] \to \underline G$ induces a quasi-isomorphism
$H_1(G(\mathbb{C}))[1] \simeq R^{Hodge}_\bZ (\underline G)$ (see
(\ref{keycomp})).}. 
Thus,  the functoriality of isomorphism  (\ref{mains})  implies that  $\Theta_{M}^{-1}$ preserves the weight filtration.  

 Let us show that $\Theta_{M}^{-1}$
preserves the Hodge filtration. A choice of a basis for $\Lambda(\bC)$,
$\Lambda = \mathbb{Z}[S]$,  yields a lifting
\begin{equation}\label{liftu}
\def\normalbaselines{\baselineskip20pt
\lineskip3pt  \lineskiplimit3pt}
\def\mapright#1{\smash{
\mathop{\to}\limits^{#1}}}
\def\mapdown#1{\Big\downarrow\rlap
{$\vcenter{\hbox{$\scriptstyle#1$}}$}}
\begin{matrix}
\underline \Lambda  &    \mapright{\tilde u= u\circ \delta} &
\mathbb{Z}_{tr}[G]  \cr
 \mapdown{\delta } & &\mapdown{Id} \cr
 \mathbb{Z}_{tr}[\Lambda] &    \mapright{u} & \mathbb{Z}_{tr}[G] .
\end{matrix}
\end{equation}
 where the homomorphism $\delta:  \underline \Lambda =
\underline  {\mathbb{Z}[S]} \to \underline {\mathbb{Z}[\Lambda]}=  \mathbb{Z}_{tr}[\Lambda]$ is defined on generators $s\in
S\subset \Lambda$ by
 the formula
\begin{equation}\label{liftuu}
 \delta(s)= [s]-[0]\in  \mathbb{Z}[\Lambda].
\end{equation}
Set $M'= cone(\tilde u)[-1]$. The diagram (\ref{liftu}) yields a
morphism $M' \to \tilde M$.
\begin{lm} The morphisms $M' \to \tilde M \to Tot_{\mathbb{Z}}(M)$ induce
$$H^0(R_{\mathbb{Z}}^{Betti}(M'))\stackrel{\alpha}{\simeq} H^0(R_{\mathbb{Z}}^{Betti}(\tilde M))\stackrel{\beta}{\simeq}
H^0(R_{\mathbb{Z}}^{Betti}\circ Tot_{\mathbb{Z}}(M)),$$
$$H^0(R^{Hodge}_\bZ(M'))\stackrel{\beta \alpha}{\simeq}
H^0(R^{Hodge}_\bZ\circ Tot_\bZ(M)).$$
\end{lm}
\begin{proof}
We have already proved in Proposition \ref{wer} that $\beta$ is an
isomorphism. Thus, it is enough to show that the composition
$\beta\alpha$ has the same property. Let us apply the functor
$R_{\mathbb{Z}}^{Betti}$ to the morphism of exact triangles
$$
\def\normalbaselines{\baselineskip20pt
\lineskip3pt  \lineskiplimit3pt}
\def\mapright#1{\smash{
\mathop{\to}\limits^{#1}}}
\def\mapdown#1{\Big\downarrow\rlap
{$\vcenter{\hbox{$\scriptstyle#1$}}$}}
\begin{matrix}
 \mathbb{Z}_{tr}[G][-1] &   \to &  M' & \to & \underline \Lambda  &
 \mapright{\tilde u} &
\mathbb{Z}_{tr}[G] \cr \mapdown{ } & &\mapdown{} && \mapdown{} &&
\mapdown{} \cr \underline G [-1]& \to &  Tot_{\mathbb{Z}}(M
)&\to & \underline \Lambda & \mapright{} & \underline G
\end{matrix}
$$
The key observation is that the connecting homomorphism
$$\tilde u_* :\mathbb{Z}[S(\bC)]=H^0(R_{\mathbb{Z}}^{Betti}( \underline \Lambda )) \to
H^0(R_{\mathbb{Z}}^{Betti}(\mathbb{Z}_{tr}[G]))=\mathbb{Z}$$ is $0$\footnote{This explains the appearance of constant term $-[0]$ in
(\ref{liftuu}).}. The rest of the proof is an easy diagram chase
combined with the formula (\ref{keycomp}).
\end{proof}
By the lemma is suffices to check that the isomorphism $$ \Theta^{\prime}=
( \Theta_{M}\circ \beta\circ  \alpha) \otimes Id: \,
H^0(R_{\mathbb{Z}}^{Betti} (M')) \otimes \bC \iso T^{Betti}_{\mathbb{Z}}(M)\otimes \bC$$
is strictly compatible with the Hodge filtration. To do this we
construct a smooth $G$-equivariant compactification $\overline G$ of $G$.
Recall that $G$ is an extension of an abelian variety $A$ by an
algebraic torus $T$. Choose an isomorphism $T\simeq {\mathbb
G}_m^r$. Then $G$, viewed as a principal $T$-bundle over $A$, gives
rise to a vector bundle $E\to A$ (which is a direct sum of $r$ line
bundles). Define
$$\overline G:= \mathbb{P}(E\oplus {\bf 1}).$$ 
The complement $\overline G - G$ is a normal crossing divisor on $\overline G$. Thus, we can view $M'$ as a complex over $\bZ[\overline {Sm}]$ (\S \ref{comp}).
The (cohomological) de Rham complex $R^*_{DR}(M')$ is identified with
$$cone(R\Gamma(\overline G, \Omega^{\cdot}_{log}) \to
 \mathbb{C}[S])$$ and the Hodge filtration on $R^*_{DR}(M')$ is
 induced by the stupid filtration on the logarithmic de Rham complex
$\Omega^{\cdot}_{log}$. In particular, $F^1H^0 \, R^*_{DR}(M')$ is
identified with the space $\Gamma(\overline G, \Omega^{1}_{log})$ of
global 1-forms on $\overline G$ with logarithmic singularities
(\cite{d2}, Theorem 3.2.5).

The translation action of $G$ on itself extends to a
$G$-action on $\overline G$.  Every global form $\omega \in
\Gamma(\overline G, \Omega^{\cdot}_{log})$ is $G$-invariant (for $G$
acts trivially on its de Rham cohomology). Conversely, every
invariant 1-form on $G$ extends to a 1-form on $\overline G$ with
logarithmic singularities (for $\dim G = \dim
F^1H_{DR}^1(G(\mathbb{C}))$). Summarizing, we get an identification
$\gamma$  of $ F^1H^0( R^*_{DR}(M'))$ with the cotangent space
$\fg^*$. The canonical pairing
$$ker\,
(\Lambda \oplus \fg \stackrel{u\oplus exp}{\longrightarrow}
G(\mathbb{C}))\otimes \fg^* \stackrel{\Theta^{\prime -1}
\otimes \gamma^{-1} }{\longrightarrow}
H^0(R_{\mathbb{Z}}^{Betti}(M'))\otimes F^1(H^0 \, R^*_{DR}(M'))\to
\mathbb{C} $$ is given by the formula
$$(\lambda \oplus \theta)\otimes \eta \, \longrightarrow \,
<\theta, \eta> . $$ It follows that, the Hodge filtration
$F^{-1}\subset H^0(R_{\mathbb{C}}^{Betti} (M'))$ is carried over by
$\Theta^{\prime}$  to the kernel of the projection
$$ker\,
(\Lambda \oplus \fg \stackrel{u\oplus exp}{\longrightarrow}
G(\mathbb{C}))\otimes _{\mathbb{Z}} \mathbb{C} \to \fg.$$

This completes the proof of (\ref{hodgec1}). 

\subsection{Proof of formula (\ref{hodgec2}).}  Set 
$$d_{\leq 1} = \Ho(Tot_\bQ \circ LAlb_\bQ): DM^{\eff}_{\gm}(\mathbb{C}, \mathbb{Q}) \to  d_{\leq 1} DM^{\eff}_{\gm}(\mathbb{C}, \mathbb{Q})\subset DM^{\eff}_{\gm}(\mathbb{C}, \mathbb{Q}),$$
$$\overline d_{\leq 1}= Ho(\overline{Tot}_\bQ \circ \overline{LAlb_\bQ} ): D^b(MHS_{\eff}^\bQ) \to \overline d_{\leq 1} D^b(MHS_{\eff}^\bQ) \subset D^b(MHS_{\eff}^\bQ), $$
$$R^{Hodge}_{\bQ, tr}= \Ho(R^{Hodge}_\bQ): DM^{\eff}_{\gm}(\mathbb{C}, \mathbb{Q})\to D^b(MHS_{\eff}^\bQ),$$
where $ d_{\leq 1} DM^{\eff}_{\gm}(\mathbb{C}, \mathbb{Q})$ (resp.  $\overline d_{\leq 1} D^b(MHS_{\eff}^\bQ)$) is the smallest strictly full subcategory of  $DM^{\eff}_{\gm}(\mathbb{C}, \mathbb{Q})$ (resp. $D^b(MHS_{\eff}^\bQ)$) that contains the image of $Tot_\bQ$ (resp.  $\overline {Tot}_\bQ$). The functor $d_{\leq 1}$  (resp. $\overline d_{\leq 1}$) is left adjoint
to $  d_{\leq 1}  DM^{\eff}_{\gm}(\mathbb{C}, \mathbb{Q}) \mono DM^{\eff}_{\gm}(\mathbb{C}, \mathbb{Q})$ (resp.  $\overline d_{\leq 1} D^b(MHS_{\eff}^\bQ) \mono D^b(MHS_{\eff}^\bQ)$).
We have to show that for every $M\in DM^{\eff}_{\gm}(\mathbb{C}, \mathbb{Q})$ the morphism
\begin{equation}\label{adjalb}
\overline d_{\leq 1}  \circ R^{Hodge}_{\bQ, tr}(M) \rar{\alpha} R^{Hodge}_{\bQ, tr} \circ d_{\leq 1}(M)
\end{equation}
  is an isomorphism.  In fact, this assertion is already proven in (\cite{bk}, Theorem 17.3.1);  the proof below is a variant of the argument given in (\cite{bk}).
  
   It suffices to show that (\ref{adjalb}) is an isomorphism in the case when $M$  is the motive of a smooth 
connected projective variety. Let $X$ be such a variety. Recall from (\cite{bk}) the
structure of the Albanese motive $ d_{\leq 1}(\mathbb{Q}_{tr}[X])$.
 Let $ X \to {\cA}_X$ be
the canonical morphism from $X$ to the extended Albanese scheme of
$X$ (\cite{r1}, Section 1). ${\cA}_X$ is a group scheme fitting into the following exact sequence 
 $$0\to {\cA}^0_X \to {\cA}_X \to \bZ \to 0,$$
where ${\cA}^0_X$ is Serre's Albanese abelian variety and $\bZ$ is viewed as the discrete group scheme over $\bC$. 
The sheaf with transfers $\underline {\cA}_X$ represented by ${\cA}_X$ defines an object of the Voevodsky
category $DM^{\eff}_{\gm, \et}(\mathbb{C}, \mathbb{Z})$.  We shall write   $\underline {\cA}_X\otimes \bQ$ for its image in  $DM^{\eff}_{\gm}(\mathbb{C}, \mathbb{Q})$.
It is clear that  $\underline {\cA}_X\otimes \bQ \subset d_{\leq 1} DM^{\eff}_{\gm}(\mathbb{C}, \mathbb{Q})$.
Consider the exact triangles 
$$P\to \mathbb{Q}_{tr}[X] \rar{u}  \underline {\cA}_X \otimes \mathbb{Q} \to P[1]$$
$$d_{\leq 1} (P) \rar{v} d_{\leq 1} (\mathbb{Q}_{tr}[X])  \rar{}  \underline {\cA}_X \otimes \mathbb{Q} \to d_{\leq 1} (P)[1],$$
where $u$ is defined as the composition $\mathbb{Q}_{tr}[X]  \to \mathbb{Q}_{tr}[{\cA}_X ] \to \underline {\cA}_X \otimes \mathbb{Q}$.
According to (\cite{bk}, Theorem 10.3.2)  the second triangle yields a commutative diagram  
$$
\def\normalbaselines{\baselineskip20pt
\lineskip3pt  \lineskiplimit3pt}
\def\mapright#1{\smash{
\mathop{\to}\limits^{#1}}}
\def\mapdown#1{\Big\downarrow\rlap
{$\vcenter{\hbox{$\scriptstyle#1$}}$}}
\begin{matrix}
    &&Hom(d_{\leq 1}(\mathbb{Q}_{tr}[X]) , \bQ(1)[2])  & \rar{v^*} & Hom(d_{\leq 1}(P), \bQ(1)[2])  \cr
  &&\mapdown{\simeq}  &  &  \mapdown{\simeq} \cr
Pic(X)& \simeq &   Hom(\mathbb{Q}_{tr}[X] , \bQ(1)[2])     & \to & NS_X \otimes \bQ,
\end{matrix}
$$
where $NS_X$ is the N\'eron-Severi group of $X$ and the map at the bottom line equals the canonical projection $Pic(X)\to NS_X$.  Moreover, the induced morphism  
  $$d_{\leq 1}(P)\to Hom(NS_X, \bQ)(1)[2]$$
  is an isomorphism.

It is enough to prove that (\ref{adjalb}) is an isomorphism for $M=P$.  By Lemma \ref{roitman} $H_iR^{Hodge}_{\bQ, tr}(P)=0$ for  $i=0, 1$ and  
$H_iR^{Hodge}_{\bQ, tr}(P)\iso H_iR^{Hodge}_{\bQ, tr}(\mathbb{Q}_{tr}[X])$ for $i>1$. It follows, that $\overline d_{\leq 1}  \circ R^{Hodge}_{\bQ, tr}(P)= \overline d_{\leq 1} (H_2(X))[2]$ is of type $(-1,-1)$. Thus, we only need to show that morphism $\alpha^*$ in the commutative diagram below is an isomorphism.
$$
\def\normalbaselines{\baselineskip20pt
\lineskip3pt  \lineskiplimit3pt}
\def\mapright#1{\smash{
\mathop{\to}\limits^{#1}}}
\def\mapdown#1{\Big\downarrow\rlap
{$\vcenter{\hbox{$\scriptstyle#1$}}$}}
\begin{matrix}
    Hom(R^{Hodge}_{\bQ, tr} \circ d_{\leq 1}(P) , \bQ(1)[2])  &  \rar{\alpha^*} & Hom(\overline d_{\leq 1}  \circ R^{Hodge}_{\bQ, tr}(P), \bQ(1)[2])  \cr
  \mapdown{\simeq}  &  &  \mapdown{\simeq} \cr
   Hom(P , \bQ(1)[2])     & \rar{R^{Hodge}_{\bQ, tr}} &   Hom(R^{Hodge}_{\bQ, tr}(P), \bQ(1)[2])  \cr
    \mapdown{\simeq}  &  &  \mapdown{\simeq} \cr
   NS_X \otimes \bQ     & \rar{\overline c_1} & Hom_{MHS^\bQ}(H_2(X), \bQ(1)),
   \end{matrix}
   $$
The next lemma shows that map $$\overline c_1: NS_X \otimes \bQ  = Pic(X)/Pic^0(X)  \otimes \bQ   \to Hom_{MHS^\bQ}(H_2(X), \bQ(1))$$ is induced by the first Chern class. Therefore,  by the Lefschetz $(1,1)$
Theorem, $\overline c_1$ and $\alpha^*$ are isomorphisms.
\begin{lm} The diagram below is commutative.
$$
\def\normalbaselines{\baselineskip20pt
\lineskip3pt  \lineskiplimit3pt}
\def\mapright#1{\smash{
\mathop{\to}\limits^{#1}}}
\def\mapdown#1{\Big\downarrow\rlap
{$\vcenter{\hbox{$\scriptstyle#1$}}$}}
\begin{matrix}
 Hom(\mathbb{Q}_{tr}[X] , \bQ(1)[2])     & \rar{R^{Hodge}_{\bQ, tr}} &   Hom (R^{Hodge}_{\bQ, tr}(\mathbb{Q}_{tr}[X]), \bQ(1)[2])  \cr
    \mapdown{\simeq}  &  &  \mapdown{} \cr
   Pic(X)  \otimes \bQ     & \rar{c_1} & Hom_{MHS^\bQ}(H_2(X), \bQ(1)),
   \end{matrix}
   $$
Here $c_1$ denotes the first Chern class map.
\end{lm}
 \begin{proof} 
 As $X$ is projective, $Pic(X)$ is generated by very ample line bundles. This reduces the proof to the  case when $X= \bP^n_\bC$. Moreover, since $H^2(\bP^n_\bC, \bQ)\iso H^2(\bP^1_\bC, \bQ)$,  it suffices to prove the commutativity of the diagram for $X=\bP^1_\bC$ and $v\in  Hom(\mathbb{Q}_{tr}[\bP^1_\bC] , \bQ(1)[2])$ being the projection to the direct summand. In this case the lemma is true by definition.   
  \end{proof}

The main theorem is proven.

 \bigskip

{\bf Notation.} $DGcat$, $\Ho(\cC)$, $\Ho(\cF)$, $\cT(\cC_1, \cC_2)$, $\underrightarrow{\cC}$, $\cC^{perf}$,  $\cC^{pretr}$,  $\bD(\cC)$, $D_{dg}(\cE)$, $C_{dg}(\cE)$ \S \ref{dgcat};  $Sm_k$, $A_{tr}[Sm_k]$,  $PSh_{tr}$, $A_{tr}[X]$, $I^{\Delta}_{tr}$,  $I^{\et}_{tr}$, $I^{et, \Delta}_{tr}$, $D\cM^{\eff}_{\et}$,  $D\cM^{\eff}_{\gm, \et}$,  $D\cM^{\eff}$, 
$DM^{\eff}_{\et}$,  $DM^{\eff}_{\gm, \et}$,  $DM^{\eff}$, $Sh^{\et}_{  tr}$
 \S \ref{dg.mot};  $A[Sm_k]$, $I^{h}_{tr}$ \S \ref{h.t};  $A[\overline {Sm}]$ \S \ref{comp};  $ C^{sing}_{A}$, $R^{Betti}_A$   \S \ref{b.r.}; $MHS^A$, $MHS^A_{\eff}$, $R^{Hodge}_A$  \S \ref{h.r.};  $C^{\Delta}$, $D_{dg}(G,A)$, $\cS_A$, $D_{dg}(G, \bZ/l^{\cdot})$, $R^{\et}_{\bZ/l^{\cdot}}$ \S \ref{e.r.};
  $\cM_1(k)$,  $D^b_{dg}(\cM_1(k))$, $Tot_A$, $LAlb_\bQ$,  $T^{Hodge}_A$,  $T^{\et}_{\bZ/l^{\cdot}}$ \S \ref{1.m.};  $\overline{LAlb_\bQ}$, $\overline{Tot_A}$  \S \ref{lalb}.

\end{document}